\documentclass[11pt,a4paper]{amsart}
\usepackage{amssymb,amsmath,epsfig,graphics,mathrsfs}

\usepackage{amsthm,epsfig,esint}

\usepackage{fancyhdr}
\usepackage{hyperref}
\hypersetup{
 colorlinks   = true,
 urlcolor     = blue,
 linkcolor    = blue,
 citecolor   = red ,
 bookmarksopen=true
}

\usepackage{amsmath}
\usepackage{amsfonts}
\usepackage{amssymb}
\usepackage{amsthm}
\usepackage{epsfig,graphics,mathrsfs}
\usepackage{graphicx}

\usepackage[usenames, dvipsnames]{color} 

\usepackage{hyperref}

\theoremstyle{plain}
\newtheorem{thrm}{Theorem}[section]
\newtheorem*{thrm*}{Theorem}
\newtheorem{lemma}[thrm]{Lemma}
\newtheorem{prop}[thrm]{Proposition}
\newtheorem{cor}[thrm]{Corollary}
\theoremstyle{definition}
\newtheorem{dfn}[thrm]{Definition}
\theoremstyle{remark}
\newtheorem{rmrk}[thrm]{Remark}
\theoremstyle{example}

\numberwithin{equation}{section}
\usepackage{color}

\begin{document}

\newcommand{\ddelta}{\delta}

\newcommand{\tx}{\tilde x}
\newcommand{\R}{\mathbb R}
\newcommand{\N}{\mathbb N}
\newcommand{\C}{\mathbb C}
\newcommand{\lie}{\mathcal G}
\newcommand{\hN}{\mathcal N}
\newcommand{\D}{\mathcal D}
\newcommand{\A}{\mathcal A}
\newcommand{\B}{\mathcal B}
\newcommand{\sL}{\mathcal L}
\newcommand{\sLi}{\mathcal L_{\infty}}

\newcommand{\G}{\Gamma}
\newcommand{\x}{\xi}

\newcommand{\eps}{\epsilon}
\newcommand{\al}{\alpha}
\newcommand{\be}{\beta}
\newcommand{\p}{\partial}  %followed by _
\newcommand{\lig}{\mathfrak}

\def\dist{\mathop{\varrho}\nolimits}

\newcommand{\BCH}{\operatorname{BCH}\nolimits}
\newcommand{\Lip}{\operatorname{Lip}\nolimits}
\newcommand{\Hol}{C}                             % Is this the right notation? $\Lambda$?
\newcommand{\lip}{\operatorname{lip}\nolimits}
\newcommand{\capQ}{\operatorname{Cap}\nolimits_Q}
\newcommand{\pCap}{\operatorname{Cap}\nolimits_p}
\newcommand{\Om}{\Omega}
\newcommand{\om}{\omega}
\newcommand{\half}{\frac{1}{2}}
\newcommand{\e}{\varepsilon}
\newcommand{\vn}{\vec{n}}
\newcommand{\X}{\Xi}
\newcommand{\tLip}{\tilde  Lip}

\newcommand{\Span}{\operatorname{span}}

\newcommand{\ad}{\operatorname{ad}}
\newcommand{\Hm}{\mathbb H^m}
\newcommand{\Hn}{\mathbb H^n}
\newcommand{\Hone}{\mathbb H^1}
\newcommand{\Lie}{\mathfrak}
\newcommand{\Layer}{V}
\newcommand{\hgrad}{\nabla_{\!H}}
\newcommand{\im}{\textbf{i}}
\newcommand{\nz}{\nabla_0}
\newcommand{\s}{\sigma}
\newcommand{\se}{\sigma_\e}

\newcommand{\ued}{u^{\e,\ddelta}}
\newcommand{\ueds}{u^{\e,\ddelta,\sigma}}
\newcommand{\tnabla}{\tilde{\nabla}}

\newcommand{\bx}{\bar x}
\newcommand{\by}{\bar y}
\newcommand{\bt}{\bar t}
\newcommand{\bs}{\bar s}
\newcommand{\bz}{\bar z}
\newcommand{\btau}{\bar \tau}

\newcommand{\LC}{\mbox{\boldmath $\nabla$}}
\newcommand{\Ne}{\mbox{\boldmath $n^\e$}}
\newcommand{\nuo}{\mbox{\boldmath $n^0$}}
\newcommand{\nuu}{\mbox{\boldmath $n^1$}}
\newcommand{\nue}{\mbox{\boldmath $n^\e$}}
\newcommand{\nuek}{\mbox{\boldmath $n^{\e_k}$}}
\newcommand{\dse}{\nabla^{H\Su, \e}}
\newcommand{\dso}{\nabla^{H\Su, 0}}
\newcommand{\tX}{\tilde X}

\newcommand\red{\textcolor{red}}
\newcommand\green{\textcolor{green}}

\newcommand{\Xie}{X^\epsilon_i}
\newcommand{\Xje}{X^\epsilon_j}
\newcommand{\Su}{\mathcal S}
\newcommand{\F}{\mathcal F}

\dedicatory{Questa nota \`e dedicata a Sandro Salsa, con grande affetto e ammirazione}

\title[Regularity for a class of quasilinear, etc.]{Regularity for a class of quasilinear degenerate parabolic equations in the Heisenberg group}

\author[L. Capogna]{L. Capogna}
\address{Luca Capogna\\Department of Mathematical Sciences, Worcester Polytechnic Institute, MA 01609\\
}

\author[G.  Citti]{G.  Citti}\address{Dipartimento di Matematica, Piazza Porta S. Donato 5, Universit\`a di Bologna, 
40126 Bologna, Italy}\email{giovanna.citti@unibo.it}

\author[N. Garofalo]{N. Garofalo}\address{Dipartimento di Ingegneria Civile e Ambientale (DICEA), Universit\`a di Padova, 35131 Padova, Italy}\email{rembrandt54@gmail.com}

\keywords{sub elliptic $p$-Laplacian, parabolic gradient estimates, Heisenberg group}

\thanks{
The first author was partially funded by NSF awards  DMS 1101478, and by a Simons collaboration grant for mathematicians 585688}

\thanks{The second author was partially funded by  
Horizon 2020 Project ref. 777822: GHAIA, and by 
PRIN 2015 Variational and perturbative aspects of nonlinear differential problems}

\thanks{The third author was supported in part by a Progetto SID (Investimento Strategico di Dipartimento) ``Non-local operators in geometry and in free boundary problems, and their connection with the applied sciences", University of Padova, 2017.}

\begin{abstract}
We extend to the parabolic setting 
some of the ideas originated with Xiao Zhong's proof in \cite{Zhong} of  the H\"older
regularity of $p-$harmonic functions in the Heisenberg group $\Hn$. Given a number $p\ge 2$, in this paper we establish the $C^{\infty}$ smoothness of weak solutions of quasilinear pde's in $\Hn$ modelled on the equation
$$\p_t u= \sum_{i=1}^{2n} X_i \bigg((1+|\nabla_0 u|^2)^{\frac{p-2}{2}} X_i u\bigg).$$
\end{abstract}

\maketitle

%\tableofcontents

\section{Introduction}

In this paper we establish the  $C^\infty$ smoothness of solutions of a certain class of quasilinear parabolic equations in the Heisenberg group $\Hn$. In a cylinder $Q=\Omega\times (0,T)$, where $\Omega\subset\Hn$ is an open set and $T>0$, we consider the equation
\begin{equation}\label{maineq}
\p_t u = \sum_{i=1}^{2n} X_i A_i(x, \nabla_0 u) \quad \quad  \text{ in  \quad }Q=\Omega\times (0,T),
\end{equation}
modeled on the regularized parabolic $p$-Laplacian 
\begin{equation}\label{plap-para} \p_t u= \sum_{i=1}^{2n} X_i \bigg((1+|\nabla_0 u|^2)^{\frac{p-2}{2}} X_i u\bigg),\end{equation}
where $p\ge 2$. The term {\it regularized} here refers to the fact that the non-linearity $(1+|\nabla_0 u|^2)^{\frac{p-2}{2}}$  affects the ellipticity of the right hand side only when the gradient blows up, and not when it vanishes, thus presenting a weaker version of the singularity in the $p-$Laplacian.
 Here, we indicate with $x=(x_1,...,x_{2n},x_{2n+1})$ the variable point in $\Hn$. We alert the reader that, although it is customary to denote the variable $x_{2n+1}$ in the center of the group with the letter $t$, we will be using $z$ instead, since we have reserved the letter $t$ for the time variable. Consequently, we will indicate with $\p_i$ partial differentiation with respect to the variable $x_i$, $i=1,...,2n$, and use the notation $Z = \p_z$ for the partial derivative $\p_{x_{2n+1}}$. The notation $\nabla_0 u \cong (X_1 u,...,X_{2n} u)$ represents the so-called \emph{horizontal gradient} of the function $u$, where
$$X_i = \partial_i - \frac{x_{n+i}}2 \partial_{z},\ \ \ \ \ \   
X_{n+i} = \partial_{n+i} + \frac{x_{i}}{2} \partial_{z},\ \ \ \  i=1,...,n.$$
As it is well-known, the $2n+1$ vector fields $X_1,...,X_{2n}, Z$ are connected by the commutation relation $[X_i,X_j] = \delta_{ij} Z$, all other commutators being trivial. 

We now introduce the relevant structural assumptions on the vector-valued function $(x,\xi)\to A(x,\xi)=(A_1(x,\xi),...,A_{2n}(x,\xi))$: there exist $p\ge 2$, $\ddelta>0$ and  $0<\lambda\le \Lambda <\infty$ such that for a.e. $x\in\Omega, \xi\in \R^{2n}$ and for all $\eta\in \R^{2n}$, one has
\begin{equation}\label{structure}
\begin{cases}
 \lambda (\ddelta+|\xi|^2)^{\frac{p-2}{2}} |\eta|^2 \le  \p_{\xi_j} A_i(x,\xi) \eta_i \eta_j \le \Lambda (\ddelta+|\xi|^2)^{\frac{p-2}{2}} |\eta|^2,
 \\
|A_i(x,\xi)| +  |\p_{x_j} A_i(x,\xi)| \le  \Lambda (\ddelta+|\xi|^2)^{\frac{p-1}{2}}.
\end{cases}
\end{equation}
Given an open set $\Omega\subset \Hn$, we indicate with $W^{1,p}(\Om)$ the Sobolev space associated with the $p$-energy $\mathscr E_{\Om,p}(u) = \frac 1p \int_{\Omega}|\nabla_0 u|^p$, i.e., the space of all functions $u\in L^p(\Omega)$ such that their distributional derivatives $X_i u$, $i=1,...,2n,$ are also in $L^p(\Omega)$. The corresponding norm
is $||u||^p_{W^{1,p}(\Omega)} = ||u||_{L^p(\Omega)}+||\nabla_0 u||_{L^p(\Omega)}.$ We denote by $W^{1,p}_0(\Om)$ the completion of $C^\infty_0(\Om)$ with respect to such norm. A function $u\in L^p((0,T), W_0^{1,p}(\Om))$
is a weak solution of \eqref{maineq} if
\begin{equation}\label{weak}
\int_0^T \int_\Omega u \phi_t\  dx dt - \int_0^T \int_\Omega \sum_{i=1}^{2n} A_i(x,\nabla_0 u) X_i \phi \ dx dt =0,
\end{equation}
for every $\phi\in C^{\infty}_0(Q)$. 
Our main result is the following.
\begin{thrm}\label{main1}  
Let $A_i$ satisfy the structure conditions \eqref{structure} for some $p\ge 2$ and $\ddelta>0$. We also assume that \eqref{maineq} can be approximated as in \eqref{structure2}-\eqref{approx1} below.
Let $u\in L^p((0,T), W_0^{1,p}(\Om))$ be a weak solution of \eqref{maineq}  in $Q=\Om\times (0,T)$.
 For any open ball $B\subset \subset \Om$ and $T>t_2\ge t_1\ge 0$, 
there exist  constants $C=C(n, p, \lambda, \Lambda, d(B, \p\Om), T-t_2, \ddelta)>0$  and 
$\al=\al (n, p, \lambda, \Lambda, d(B, \p\Om), T-t_2, \ddelta) \in (0,1)$ such that
\begin{equation}\label{c1alpha}
||\nabla_0 u ||_{C^{\alpha} (B\times (t_1,t_2) )} + ||Zu||_{C^{\alpha} (B\times (t_1,t_2) )}\le C
\bigg(\int_0^T \int_\Om (\ddelta+|\nabla_0 u|^2)^{\frac{p}{2}} dx dt \bigg)^{\frac{1}{p}} .
\end{equation}
\end{thrm}

Besides the structural hypothesis \eqref{structure}, Theorem \ref{main1} will be established under an additional technical approximating assumption. Namely, for $\e\ge 0$ we consider the left-invariant Riemannian metric $g_\e$ in $\Hn$ in which the frame defined by $X^\e_1 = X_1,...,X^\e_{2n}= X_{2n}, X_{2n+1}^\e = \e Z$ is orthonormal, and denote by $\nabla_\e$ the gradient in such metric. We will adopt the unconventional notation $W^{1,p,\e}(\Om)$ to indicate the Sobolev space associated with the $p$-energy $\mathscr E_{\Om,p,\e}(u) = \frac 1p \int_{\Omega}|\nabla_\e u|^p$. We assume that one can approximate $A_i$ by  a $1$-parameter family of regularized approximants $A^\e(x,\xi)=(A_1^\e(x,\xi),...,A_{2n+1}^\e(x,\xi))$  defined for a.e. $x\in \Om$ and every $\xi \in \R^{2n+1}$, and such that 
 for a.e. $x\in\Omega, $ and for all $\xi = (\xi_1,...,\x_{2n},\xi_{2n+1})\in \R^{2n+1}$  one has uniformly on compact subsets of $\Om \times (0,T)$,
\begin{equation}\label{structure2}
(A_1^\e(x,\xi),...,A_{2n+1}^\e(x,\xi))\ \underset{\e\to 0^+}{\longrightarrow}\    (A_1(x, \xi_1,...,\x_{2n}), ..., A_{2n} (x, \xi_1,...,\x_{2n}),0),
\end{equation}
and furthermore
\begin{equation}\label{structure-epsilon}
\begin{cases}
\lambda (\ddelta+|\xi|^2)^{\frac{p-2}{2}} |\eta|^2 \le  \p_{\xi_j} A_i^\e(x,\xi) \eta_i \eta_j \le \Lambda (\ddelta+|\xi|^2)^{\frac{p-2}{2}} |\eta|^2, 
\\
|A_i^\e(x,\xi)| +  |\p_{x_j} A_i^\e(x,\xi)| \le  \Lambda (\ddelta+|\xi|^2)^{\frac{p-1}{2}},
\end{cases}
\end{equation}
 for all $\eta\in \R^{2n+1}$, and 
 for some $0<\lambda\le \Lambda <\infty$ independent of $\e$.
The proof of the $C^{1,\alpha}$  regularity in Theorem \ref{main1} is based on a priori estimates for solutions of the one-parameter family of regularized partial differential equations which approximate \eqref{maineq} as the parameter $\e\to 0$. The key will be in establishing estimates that do not degenerate as $\e\to 0$. Specifically, for any $\e>0$ we will consider a weak solution $u^\e$ to  the equation
\begin{equation}\label{approx1}
\p_t u^\e = \sum_{i=1}^{2n+1} X^\e_i A_i^\e (x, \nabla_\e u^\e)
\end{equation}
in a cylinder $Q_0=B(x_0,R_0)\times (t_0,t_1)$, with  $B(x_0,R_0)\subset \Omega$ and $(t_0,t_1)\subset (0,T)$,  and  with (parabolic) boundary data $u^\e=u$. Since \eqref{approx1} is strongly parabolic for every $\e>0$, the solutions $u^\e$ are smooth in every compact subset $K\subset Q_0$ and, in view of the comparison principle, and of the uniform Harnack inequality established in \cite{ACCN},  converge uniformly on compact subsets   to  a function $u_0$. The bulk of the paper consists in establishing  higher regularity estimates for $u_\e$ that are uniform in $\e>0$, to show that $u_0$ inherits such higher regularity and is a solution of \eqref{maineq}, thus it coincides with $u$. Here is our main result in this direction.

\begin{thrm}\label{main epsilon}
In the hypothesis  \eqref{structure2}, \eqref{structure-epsilon},
consider  for each $\e>0$ a weak solution $u^\e\in L^p ((0,T), W^{1,p,\e}(\Om))\cap C^2(Q)$  of the approximating equation \eqref{approx1} in $Q$.
For any open ball $B\subset \subset \Om$ and $T>t_2\ge t_1\ge 0$
there exists a constant $C=C(n, p, \lambda, \Lambda, d(B, \p\Om), T-t_2, \ddelta)>0$, such that
\begin{equation}\label{mainestepsilon}
||\nabla_\e u ^\e ||_{L^\infty(B\times (t_1,t_2) )}^{ p}+ \int_{t_1}^{t_2} \int_B (\ddelta+|\nabla_\e u^\e
|^2)^{\frac{p-2}{2}} \sum_{i,j=1}^{2n} |X^\e_i X^\e_j u^\e |^2 dx dt \le C \int_0^T \int_\Om (\ddelta+|\nabla_\e u^\e|^2)^{\frac{p}{2}} dx dt.
\end{equation}
Moreover, for any open ball $B\subset \subset \Om$ and $T>t_2\ge t_1\ge 0$, 
there exist  constants $C>0$ and  $\al \in (0,1)$, which depend on $n, p, \lambda, \Lambda, d(B, \p\Om), T-t_2, \ddelta$, such that
\begin{equation}\label{c1alpha0}
||\nabla_\e u^\e ||_{C^{\alpha} (B\times (t_1,t_2) )} + ||Zu^\e||_{C^{\alpha} (B\times (t_1,t_2) )}\le C
\bigg(\int_0^T \int_\Om (\ddelta+|\nabla_\e u^\e|^2)^{\frac{p}{2}} dx dt \bigg)^{\frac{1}{p}}.
\end{equation}
\end{thrm}

We emphasise that the constants in \eqref{mainestepsilon} and \eqref{c1alpha0} are independent of $\e$.

\bigskip

It is worth mentioning here that the prototype for the class of equations \eqref{maineq} and for their parabolic approximation comes from considering the regularized $p-$Laplacian operator $L_p u= \text{div}_{g_0,\mu_0}( (\ddelta+|\nabla_0 u|_{g_0}^2)^{\frac{p-2}{2}} \nabla_0 u)$ in a sub-Riemannian contact manifold $(M,\omega,g_0)$, where $M$ is the underlying differentiable manifold, $\omega$ is the contact form and $g_0$ is a Riemannian metric on the contact distribution. The measure $\mu_0$ is the corresponding Popp measure. The approximants are constructed through Darboux coordinates, considering the $p-$Laplacians associated to a family of Riemannian metrics $g_\e$ that tame $g_0$ and such that the metric structure of the spaces $(M,g_\e)$ converge in the Gromov-Hausdorff sense to the metric structure of $(M, \omega,  g_0)$. For a more detailed description, see \cite[Section 6.1]{CCLDO}.  As an immediate corollary of Theorem \ref{main1} one has the following.

\begin{thrm}\label{main3}  Let $(M,\omega,g_0)$ be a contact, sub-Riemannian manifold and let $\Om\subset M$ be an open set. For $p\ge 2$, 
consider $u\in L^p((0,T), W_0^{1,p}(\Om))$ be a weak solution of 
$$\p_t u =\operatorname{div}_{g_0,\mu_0}( (\ddelta+|\nabla_0 u|_{g_0}^2)^{\frac{p-2}{2}} \nabla_0 u) ,$$ in $Q=\Om\times (0,T)$.
 For any open ball $B\subset \subset \Om$ and $T>t_2\ge t_1\ge 0$, 
there exist  constants $C=C(n, p, d(B, \p\Om), T-t_2, \ddelta)>0$  and 
$\al=\al (n, p,  d(B, \p\Om), T-t_2, \ddelta) \in (0,1)$ such that
\begin{equation}\label{c1alpha}
||\nabla_0 u ||_{C^{\alpha} (B\times (t_1,t_2) )} + ||Zu||_{C^{\alpha} (B\times (t_1,t_2) )}\le C
\bigg(\int_0^T \int_\Om (\ddelta+|\nabla_0 u|^2)^{\frac{p}{2}} dx dt \bigg)^{\frac{1}{p}} .
\end{equation}
\end{thrm}

The $C^{1,\al}$ estimates in \eqref{c1alpha0} in Theorem \ref{main epsilon}  allow us to apply the Schauder theory developed in \cite{bramanti-brandolini, xu}, and finally deduce the following result.

\begin{thrm}\label{main}   
Let $k\in \N$ and $\al\in (0,1)$. If $A_i(x,\xi), \p_{x_k} A_i(x,\xi), \p_{\xi_j} A_i(x,\xi) \in C^{k,\al}_{loc}$ satisfy the structure conditions \eqref{structure} for some $p\ge 2$ and $\ddelta>0$, then any weak solution $u\in L^p((0,T), W_0^{1,p}(\Om))$ is $C^{k+1,\al}$ on compact subsets of $Q$.
\end{thrm}

The present paper is the first study of higher regularity of weak solutions for the non stationary $p$-Laplacian type in the sub-Riemannian setting, and it is based on the techniques introduced by Zhong in \cite{Zhong}. The stationary case has been developed so far essentially  only in the Heisenberg group case thanks to the work of Domokos, \cite{Dom}, Manfredi, Mingione \cite{Manfredi-Mingione}, Mingione, Zatorska-Goldstein and Zhong \cite{Mingione-Zatorska-Zhong}, Ricciotti \cite{Ricciotti}, \cite{Ricciotti1} and Zhong \cite{Zhong}. Regularity in more general  contact sub-Riemannian manifolds, including the rototraslation group,  has been recently established by the two of the authors and coauthors \cite{CCLDO} and independently by Mukherjee 
\cite{M} based on an extension of the techniques in \cite{Zhong}. Domokos and Manfredi \cite{DomMan} are rapidly making substantial progress in higher steps groups and in some special non-group structures, using the Riemannian approximation approach as in the work \cite{CCLDO}.

The plan of the paper is as follows. In Section \ref{S:prelim} we collect some preparatory material that will be used in the main body of the paper. Section \ref{S:cacc} is devoted to proving the first part of Theorem \ref{main epsilon}, which establishes the Lipschitz regularity of the approximating solutions $u^\e$. In Section \ref{S:hold} we prove the H\"older regularity of derivatives of $u^\e$ in Theorem \ref{main epsilon}. Finally, in Section \ref{S:main} we use the comparison principle and  Theorem \ref{main epsilon} to establish Theorem \ref{main1}.

\bigskip

Some final comments are in order. The non-degeneracy hypothesis $\delta>0$ in \eqref{structure} (see also \eqref{structure-epsilon}) is not needed in the Euclidean setting and, in the  stationary regime, it is not needed in the Heisenberg group either.  We suspect the $C^{1,\alpha}$ regularity of weak solutions for \eqref{main1} still holds without this hypothesis, but at the moment we are unable to prove it. In this note we use $\delta>0$ notably in \eqref{need delta>0 } and in Theorem \ref{moser}. 

In order to extend the parabolic regularity theory to the sub-Riemannian setting   one has to find a way to implement, in this  non-Euclidean framework, some of the techniques introduced by Di Benedetto \cite{DB} which rely on non-isotropic cylinders in space-time.   The key idea  is to work with cylinders whose dimensions are
suitably rescaled to reflect the degeneracy exhibited by the partial differential equation. To give an example, if one sets $x\in \Om$, $R,\mu>0$, one can define the intrinsic cylinder
$$Q_R(\mu):= B(x,R)\times (-\mu^{2-p}R^2,0), \text{ with }\sup_{Q_R(\mu)}|\nabla_0 u|\le \mu.$$
In contrast with the usual parabolic cylinders of the linear theory, the shape of the $Q_R(\mu)$ cylinders is stretched in the time dimension by a factor of the order $|\nabla_0 u|^{2-p}.$ 

The use of such non-isotropic cylinders  seems  necessary in order to make-up for the different homogeneity of the time derivative and the space derivatives in the degenerate regime $\delta=0$. In a future study we plan to return to the problem of extending Di Benedetto's Caccioppoli inequalities on non-isotropic cylinders to the Heisenberg group and beyond. 

\medskip

\noindent{\it Acknowledgements.}  We thank Vira A. Markasheva, who collaborated with us on an earlier version of this project.

\section{Preliminaries}\label{S:prelim}

In this section we collect a few definitions and preliminary results that will be used throughout the rest of the paper.  As indicated in the introduction, for each $\e\in (0,1)$ we define $g_\e$ to be the Riemannian metric in $\Hn$ such that $X_1,...,X_{2n}, \e Z$ is an orthonormal frame, and denote such frame as $X_1^\e, ..., X_{2n+1}^\e$. The corresponding gradient operator will be denoted by 
$\nabla_\e $.

\begin{dfn}\label{cylinder}
For $x_0\in \Om\subset \Hn$, we define a  parabolic cylinder $Q_{\e, r}(x_0, t_0)\subset Q$
to be a set of the form
$Q_{\e, r}(x_0, t_0)= B_{\e}(x_0, r) \times (t_0-r^2, t_0 ).$
where $r>0$, $B_{\e}(x_0, r)\subset \Omega$ denotes the $g_\e$-Riemannian ball of center $x_0$
and $t_0\in (0, T)$.  We call parabolic boundary of the cylinder $Q_{\e, r}(x_0, t_0)\subset Q$
the  set $B_{\e}(x_0, r) \times \{t_0-r^2\} \cup \partial B_{\e}(x_0, r) \times [t_0-r^2, t_0).$
\end{dfn}

First of all we recall the H\"older regularity, and local boundedness of weak solutions of \eqref{main1} and \eqref{approx1} from \cite{ACCN}.

\begin{lemma}  Let $Q=\Om\times (0,T)\subset \Hn\times \R^+$, and $\delta\ge 0$. For  $\e\ge 0$ and $p\ge 2$, consider a weak solution $u^\e\in L^p ((0,T), W^{1,p,\e}(\Om))\cap C^2(Q)$  of the approximating equation \eqref{approx1} in $Q$.
For any open ball $B\subset \subset \Om$ and $T>t_2\ge t_1\ge 0$
there exist  constants $C=C(n, p, \lambda, \Lambda, d(B, \p\Om), T-t_2)>0$,  and  $\al=\al (n, p, \lambda, \Lambda, d(B, \p\Om), T-t_2) \in (0,1),$ such that
\begin{equation}\label{alpha}
||u^\e ||_{C^{\alpha} (B\times (t_1,t_2) )}\le C
\bigg(\int_0^T \int_\Om (\ddelta+|\nabla_\e u^\e|^2)^{\frac{p}{2}} dx dt \bigg)^{\frac{1}{p}}.
\end{equation}
\end{lemma}
When $\e>0$ and $\delta>0$, classical regularity results (e.g., \cite{LadUr1967}) yield that weak solutions have bounded gradient, and hence \eqref{approx1} is strongly parabolic, thus leading to weak solutions being smooth. Clearly such smoothness may degenerate as $\e\to 0$, and the main point of this paper is to show that this does not happen.

Let  $\Om\subset\R^n$ be  a bounded open set and let $Q=\Om\times (0,T)$. For a function $u:Q\to \R$, and $1\le p,q$ we define the Lebesgue spaces $L^{p,q}(Q)=L^q([0,T], L^p(\Om))$, endowed with the norms
\begin{equation}\label{lpq-norms}
||u||_{L^{p,q}(Q)} =\Big(\int_0^T (\int_\Om |u|^p dx)^{\frac{q}{p}} dt\Big)^{\frac{1}{q}}.
\end{equation}
When $p=q$, we will refer to $L^{p,p}(Q)$ as $L^p(Q)$.
 One has the following useful reformulation of the Sobolev embedding theorem \cite{fs:hardy} in terms of $L^{p,q}$ spaces. In the next statement, we will indicate with $N=2n+2$ the homogenous dimension of $\Hn$ with respect to the non-isotropic group dilations, and we will denote by $N_1=N+2=2n+4$ the corresponding parabolic dimension with respect to the dilations $(x,t)\to (\delta_\lambda x,\lambda^2 t)$.
 
\begin{lemma}\label{lemma-sobolev}
Let $v$ be Lipschitz function in $Q$, and assume that for all $0<t<T$,  $v(\cdot, t)$ has  compact support in $\Om\times\{t\}$. 
\begin{itemize}
\item[(i)] There exists $C=C(n)>0$ such that for any $\e \in [0,1]$ one has  $$||v||_{L^{\frac{2N}{N-2}, 2}(Q)} \le C ||\nabla_\e v||_{L^{2,2}(Q)}.$$
\item[(ii)] If $v\in L^{2,\infty}(Q)$, then $v\in L^{\frac{2N_1}{N_1-2},\frac{2N_1}{N_1-2}}(Q)$, and there exists $C>0$, depending on $n$, such that
for any $\e \in [0,1]$ one has 
$$||v||_{L^{\frac{2N_1}{N_1-2},\frac{2N_1}{N_1-2}}(Q)}^2 \le C (||v||^2_{L^{2,\infty}(Q)}+||\nabla _\e v||_{L^{2,2}(Q)}^2).$$ 
\end{itemize}
\end{lemma}
We note that as $\e$ decreases to zero, the background geometry shifts from Riemannian to sub-Riemannian. The stability with respect to $\e$  of the constant $C$ in the Lemma \ref{lemma-sobolev} is not trivial, see \cite{CM, CC}.

In the sequel we will use an interpolation inequality that will take the place of the Sobolev inequality in a Moser type iteration, just as, for example, in \cite[Proposition 4.2]{CLM}. Although the result does not use the equation at all, we state it in terms that will make it immediately applicable later on. Henceforth, to simplify the notation, we will routinely omit the indication of $dx$, $dx dt$, etc. in all integrals involved, unless there is risk of confusion.  

\begin{lemma}\label{SobXU}
Let $u^\e$ be a weak solution of \eqref{approx1} in $Q$.  
If $ \beta \geq 0$,  and $\eta\in C^{1}([0,T], C^\infty_0 (\Om))$ vanishes on the 
parabolic boundary of $Q$, then there is a constant $C>0$, depending only 
on $||u^\e||_{L^\infty(Q)}$, such that
\begin{align*}
& \int_{t_1} ^{t_2}\int_\Omega
(\ddelta + |\nabla_\e u^\e|^2)^{(\beta+p+2)/2} |\eta|^{\beta+2}
 \leq 
C (\beta +p +1)^2 \int_{t_1} ^{t_2}\int_\Omega
(\ddelta+  |\nabla_\e u^\e|^2)^{\frac{p+\beta-2}{2}} \sum_{i,j=1}^{2n+1}| X_j^\e X_i^\e u^\e|^2 |\eta|^{\beta + 2}
\\ 
&+ C\beta^2 \int_{t_1} ^{t_2}\int_\Omega
(\ddelta + |\nabla_\e u^\e|^2)^{(\beta+p)/2} |\eta|^{\beta} ( |\eta|^{2}
+ |\nabla_\e \eta |^2).
\end{align*}
\end{lemma}

\begin{proof} Writing $(\ddelta + |\nabla_\e u^\e|^2)^{(\beta+p+2)/2}=(\ddelta + |\nabla_\e u^\e|^2)^{(\beta+p)/2}(\ddelta + |\nabla_\e u^\e|^2)$, one has 
\begin{align*}
& \int_{t_1} ^{t_2}\int_\Omega
(\ddelta + |\nabla_\e u^\e|^2)^{(\beta+p+2)/2} |\eta|^{\beta+2} =
\ddelta 
\int_{t_1} ^{t_2}\int_\Omega
(\ddelta + |\nabla_\e u^\e|^2)^{(\beta+p)/2} 
|\eta|^{\beta+2}
\\
& +  
\sum_{i=1}^{2n+1}\int_{t_1} ^{t_2}\int_\Omega
(\ddelta + |\nabla_\e u^\e|^2)^{(\beta+p)/2} 
X_i^\e u^\e X_i^\e u^\e |\eta|^{\beta+2}
 = \ddelta 
\int_{t_1} ^{t_2}\int_\Omega
(\ddelta + |\nabla_\e u^\e|^2)^{(\beta+p)/2} 
|\eta|^{\beta+2}
\\
& - \sum_{i=1}^{2n+1}\int_{t_1} ^{t_2}\int_\Omega
X_i^\e\Big((\ddelta + |\nabla_\e u^\e|^2)^{(\beta+p)/2} X_i^\e u^\e\Big)  u^\e |\eta|^{\beta+2} 
- (\beta+2)\int_{t_1} ^{t_2}\int_\Omega
(\ddelta + |\nabla_\e u^\e|^2)^{(\beta+p)/2} X_i^\e u^\e u^\e |\eta|^{\beta+1}X_i^\e\eta 
\\
& \leq\ddelta 
\int_{t_1} ^{t_2}\int_\Omega
(\ddelta + |\nabla_\e u^\e|^2)^{(\beta+p)/2} 
|\eta|^{\beta+2}
 +(\beta+p+1)\int_{t_1} ^{t_2}\int_\Omega
(\ddelta + |\nabla_\e u^\e|^2)^{(\beta+p)/2} \sum_{i,j=1}^{2n+1}|X_i^\e X_i^\e u^\e|   |u^\e| |\eta|^{\beta+2} 
\\
& + C(\beta+2)\int_{t_1} ^{t_2}\int_\Omega
(\ddelta + |\nabla_\e u^\e|^2)^{(\beta+p +1)/2}  |\eta|^{\beta+1}|\nabla_\e\eta|. 
\notag
\end{align*}
To conclude the argument, it suffices to apply Young's inequality.

\end{proof}

\section{Caccioppoli type inequalities and Lipschitz regularity of $u^\e$.}\label{S:cacc}

In this section we establish Lipschitz regularity for the derivatives of the solutions  $u^\e$.  The main results of this section are summarized in the following estimates,  which are unform in $\e>0$.

\begin{thrm}\label{main-epsilon-lebesgue} 

Let $A^\e_i$ satisfy the structure conditions \eqref{structure} for some $p\ge 2$ and $\ddelta>0$. Consider an open set $\Om\subset \Hn$  and $T> 0$, and let $u^\e$ be a   weak solution of
\eqref{approx1}  in $Q=\Om\times (0,T)$. 
For any open ball $B\subset \subset \Om$ and $T>t_2\ge t_1\ge 0$, 
there exists a constant $C>0$, depending on $n, p, \lambda, \Lambda, d(B, \p\Om), T-t_2, \ddelta$,  such that
\begin{align}\label{mainest}
& ||\nabla_\e u^\e  ||_{L^\infty(B\times (t_1,t_2) )}^{ p}+ \int_{t_1}^{t_2} \int_B (\ddelta+|\nabla_\e u|^2)^{\frac{p-2}{2}} \bigg(\sum_{i,j=1}^{2n} |X_i^\e X_j^\e u^\e |^2  + |\nabla_\e Zu^\e|\bigg) 
\\
& \le C \int_0^T \int_\Om (\ddelta+|\nabla_\e u^\e|^2)^{\frac{p}{2}}.
\notag
\end{align}
\end{thrm}

The proof of Theorem \ref{main-epsilon-lebesgue} will follow from combining the results in Theorem \ref{CCGMt2.8}, Lemma \ref{GainDu}, Proposition \ref{moser} and Proposition \ref{Zu-reg}, that are all established later in the section.
The  Caccioppoli inequalities needed to prove Theorem \ref{mainest} will take up most of the section, and they all  apply to a solution $u^\e$ of the approximating equation \eqref{approx1} in a cylinder $ Q=\Omega\times (0,T)$.
We begin with two lemmas in which we explicitly detail the pde's satisfied by the smooth approximants $X_\ell^\e u^\e$ and $Zu^\e$.

\begin{lemma} \label{diff-PDE} Let $u^\e$ be a solution of \eqref{approx1} in $Q$. If we set $v^\e_\ell = X^\e_\ell u^\e$, with $ \ell = 1, . . . , 2n$, and $s_\ell = (-1)^{[\ell/n]}$, then the function $v^\e_\ell$ solves the equation
\begin{align}\label{eqderivX}
& \p_t v^\e_\ell
 = \sum_{i,j=1}^{2n+1} X^\e_i\Big( A_{i, \xi_j}^\e (x, \nabla_\e u^\e) X^\e_\ell X^\e_ju^\e \Big) 
\\
& + \sum_{i=1}^{2n+1} X^\e_i\Big( A_{i, x_\ell}^\e (x, \nabla_\e u^\e) - \frac{s_\ell x_{\ell+s_{\ell} n}}{2} A_{i, x_{2n+1}}^\e (x, \nabla_\e u^\e) \Big) +
s_\ell Z (A_{\ell+ s_{\ell}n}^\e (x, \nabla_\e u^\e)).
\notag
\end{align}
\end{lemma}

\begin{proof}
Differentiating \eqref{approx1} with respect to $X^\e_\ell$, when $\ell\leq n$, we find
\begin{align*}
& \p_t v^\e_\ell = \sum_{i=1}^{2n+1} X^\e_\ell X^\e_i A_i^\e (x, \nabla_\e u^\e) = \sum_{i=1}^{2n+1} X^\e_i\Big(X^\e_\ell  A_i^\e (x, \nabla_\e u^\e) \Big)
+ \sum_{i=1}^{2n+1} [X^\e_\ell, X^\e_i] A_i^\e (x, \nabla_\e u^\e)
\\
& = 
\sum_{i,j=1}^{2n+1} X^\e_i\Big( A_{i, \xi_j}^\e (x, \nabla_\e u^\e) X^\e_\ell X^\e_ju^\e \Big)+ 
\sum_{i=1}^{2n+1} X^\e_i\Big( A_{i, x_\ell}^\e (x, \nabla_\e u^\e) - \frac{x_{\ell+n}}{2} A_{i, x_{2n+1}}^\e (x, \nabla_\e u^\e) \Big) + Z (A_{\ell+n}^\e (x, \nabla_\e u^\e)).
\end{align*}
Taking the derivative with respect to $X^\e_{\ell}$ when $\ell\geq n+1$, we obtain
\begin{align*}
& \p_t v^\e_{\ell} = \sum_{i=1}^{2n+1} X^\e_{\ell} X^\e_i A_i^\e (x, \nabla_\e u^\e)=  
    \sum_{i=1}^{2n+1} X^\e_i\Big(X^\e_{\ell} A_i^\e (x, \nabla_\e u^\e) \Big)
+ \sum_{i=1}^{2n+1} [X^\e_{\ell}, X^\e_i] A_i^\e (x, \nabla_\e u^\e)
\\
& = 
\sum_{i,j=1}^{2n+1} X^\e_i\Big( A_{i, \xi_j}^\e (x, \nabla_\e u^\e) X^\e_{l} X^\e_ju^\e \Big)
 + 
\sum_{i=1}^{2n+1} X^\e_i\Big( A_{i, x_{\ell}}^\e (x, \nabla_\e u^\e) + \frac{x_{\ell-n}}{2} A_{i, x_{2n+1}}^\e (x, \nabla_\e u^\e) \Big)
- Z (A_{\ell-n}^\e (x, \nabla_\e u^\e)).
\end{align*}

\end{proof}

\begin{lemma}\label{eqZ} 
Let $u^\e$ be a solution of \eqref{approx1} in  $Q$. Then, the function
$Zu^\e$  is a solution of the equation
$$\partial_t Zu^\e=
\sum_{i,j=1}^{2n+1}
X^\e_i(A^\e_{i, \xi_j}(x, \nabla_\e u^\e)X^\e_j Zu^\e) +
\sum_{i=1}^{2n+1}
X^\e_i(A^\e_{i, x_{2n+1}}(x, \nabla_\e u^\e)). 
$$
\end{lemma}

\begin{proof}
It follows by a straightforward differentiation of \eqref{approx1} with respect to $Z$, and we omit the details.

\end{proof}

\begin{lemma}\label{stimaZu} 
Let $u^\e$ be a solution of \eqref{approx1} in $Q$.
For any $ \beta \geq 0$ and for all $\eta\in C^{1}([0,T], C^\infty_0 (\Om))$,
one has
\begin{align*}
& \frac{1}{\beta + 2}\int_\Omega |Zu^\e|^{\beta+2} \eta^2 \Big|_{t_1}^{t_2} +\frac{\lambda
(\beta + 1)}{2}
\int_{t_1} ^{t_2}\int_\Omega
(\ddelta+|\nabla_\e u^\e|^2)^{\frac{p-2}{2}} | \nabla_\e Zu^\e|^2|Zu^\e|^{\beta} |\eta|^2,
\\
& \le\frac{\Lambda}{
(\beta + 1)}\Big(\frac{16 \Lambda}{\lambda} +2 \Big)
\int_{t_1} ^{t_2}\int_\Omega
(\ddelta+|\nabla_\e u^\e|^2)^{\frac{p-2}{2}}|\nabla_\e \eta|^2 |Zu^\e|^{\beta+2}
+ \frac{2} {\beta + 2}\int_{t_1}^{t_2}\int_\Omega |Zu^\e|^{\beta+2} \eta \partial_t \eta  
\\
& + \Lambda
(\beta + 1) \Big(\frac{16 \Lambda}{\lambda} +2 \Big)\int_{t_1} ^{t_2}\int_\Omega 
(\ddelta+|\nabla_\e u^\e|^2)^{\frac{p}{2}} \eta^2  |Zu^\e|^{\beta}. 
\end{align*}
\end{lemma}

\begin{proof}
We use $\phi= \eta^2 |Zu^\e|^{\beta}Zu^\e$  as a test function  in the equation satisfied by $Zu^\e$, see Lemma \ref{eqZ}, to obtain 
\begin{align*}
& \int_{t_1} ^{t_2}\int_\Omega \partial_t Zu^\e \eta^2 |Zu^\e|^{\beta}Zu^\e =\int_{t_1} ^{t_2}\int_\Omega \sum_{i,j=1}^{2n+1}
X^\e_i(A^\e_{i, \xi_j}(x, \nabla_\e u)X^\e_j Zu^\e) \eta^2 |Zu^\e|^{\beta}Zu^\e 
\\
&+\int_{t_1} ^{t_2}\int_\Omega 
\sum_{i,j=1}^{2n+1}
X^\e_i(A^\e_{i, x_{2n+1}}(x, \nabla_\e u)) \eta^2 |Zu^\e|^{\beta}Zu^\e.
\end{align*}
The left-hand side of the latter equation can be expressed as follows: 
$$\int_{t_1} ^{t_2}\int_\Omega \partial_t Zu^\e \eta^2 |Zu^\e|^{\beta}Zu^\e=
\frac{1}{\beta + 2}\int_{t_1} ^{t_2}\int_\Omega \partial_t |Zu^\e|^{\beta+2} \eta^2. 
$$
Considering the first term in the right-hand side, we obtain
\begin{align*}
& \int_{t_1} ^{t_2}\int_\Omega \sum_{i,j=1}^{2n+1}
X^\e_i(A^\e_{i, \xi_j}(x, \nabla_\e u)X^\e_j Zu^\e) \eta^2 |Zu^\e|^{\beta}Zu^\e = - \int_{t_1} ^{t_2}\int_\Omega \sum_{i,j=1}^{2n+1}
A^\e_{i, \xi_j}(x, \nabla_\e u)X^\e_j Zu^\e 
X^\e_i (\eta^2 |Zu^\e|^{\beta}Zu^\e)
\\
& =  - 2\int_{t_1} ^{t_2}\int_\Omega \sum_{i,j=1}^{2n+1}
A^\e_{i, \xi_j}(x, \nabla_\e u)X^\e_j Zu^\e 
\eta X^\e_i \eta  |Zu^\e|^{\beta}Zu^\e - (\beta +1)\int_{t_1} ^{t_2}\int_\Omega \sum_{i,j=1}^{2n+1}
A^\e_{i, \xi_j}(x, \nabla_\e u)X^\e_j Zu^\e 
\eta^2 |Zu^\e|^{\beta} X^\e_i Zu^\e.
\end{align*}
As for the second term in the right-hand side, we have
\begin{align*}
& \int_{t_1} ^{t_2}\int_\Omega 
\sum_{i,j=1}^{2n+1}
X^\e_i(A^\e_{i, x_{2n+1}}(x, \nabla_\e u)) \eta^2 |Zu^\e|^{\beta}Zu^\e 
= - 2 \int_{t_1} ^{t_2}\int_\Omega 
\sum_{i,j=1}^{2n+1}
A^\e_{i, x_{2n+1}}(x, \nabla_\e u) \eta X^\e_i \eta |Zu^\e|^{\beta}Zu^\e
\\  
& - (\beta + 1)
\int_{t_1} ^{t_2}\int_\Omega 
\sum_{i,j=1}^{2n+1}
A^\e_{i, x_{2n+1}}(x, \nabla_\e u) \eta^2 |Zu^\e|^{\beta} X^\e_i Zu^\e.  
\end{align*}
Combining the latter three equations, we find
\begin{align*}
& \frac{1}{\beta + 2}\int_{t_1} ^{t_2}\int_\Omega \partial_t |Zu^\e|^{\beta+2} \eta^2  +
(\beta + 1)
\int_{t_1} ^{t_2}\int_\Omega
\sum^{2n+1}_{
i,j=1}\p_{\xi_j}A^\e_i(x, \nabla_\e u^\e)X^\e_jZu^\e
\eta^2|Zu^\e|^\beta  X^\e_iZu^\e 
\\
& = -2\int_{t_1}^{t_2}
\int_\Omega\sum^{2n+1}_{
i,j=1}\p_{\xi_j}A^\e_i(x, \nabla_\e u^\e)X^\e_jZu^\e X^\e_i\eta
\eta|Zu^\e|^\beta Zu^\e 
 - 2 \int_{t_1} ^{t_2}\int_\Omega 
\sum_{i,j=1}^{2n+1}
A^\e_{i, x_{2n+1}}(x, \nabla_\e u) \eta X^\e_i \eta |Zu^\e|^{\beta}Zu^\e  
\\
& - (\beta + 1)
\int_{t_1} ^{t_2}\int_\Omega 
\sum_{i,j=1}^{2n+1}
A^\e_{i, x_{2n+1}}(x, \nabla_\e u) \eta^2 |Zu^\e|^{\beta} X^\e_i Zu^\e.
\end{align*}
The structure conditions \eqref{structure} yield
\begin{align*}
& \frac{1}{\beta + 2}\int_\Omega |Zu^\e|^{\beta+2} \eta^2\Big|_{t_1}^{t_2} +\lambda
(\beta + 1)
\int_{t_1} ^{t_2}\int_\Omega
(\ddelta+|\nabla_\e u^\e|^2)^{\frac{p-2}{2}} |\nabla_\e Zu^\e|^2 |Zu^\e|^{\beta} |\eta|^2\le
\\
& \frac{1}{\beta + 2}\int_\Omega |Zu^\e|^{\beta+2} \eta^2 \Big|_{t_1}^{t_2} +
(\beta + 1)
\int_{t_1} ^{t_2}\int_\Omega
\sum^{2n+1}_{
i,j=1}\p_{\xi_j} A^\e_i(x, \nabla_\e u^\e)X^\e_jZu^\e
 X^\e_iZu^\e \;   \eta^2|Zu^\e|^\beta 
\\
& = -2\int_{t_1}^{t_2}
\int_\Omega\sum^{2n+1}_{
i,j=1}\p_{\xi_j}A^\e_i(x, \nabla_\e u^\e)X^\e_jZu^\e X^\e_i\eta
\eta|Zu^\e|^\beta Zu^\e + \frac{2} {\beta + 2}\int_{t_1}^{t_2}\int_\Omega |Zu^\e|^{\beta+2} \eta \partial_t \eta 
\\
& - 2 \int_{t_1} ^{t_2}\int_\Omega 
\sum_{i=1}^{2n+1}
A^\e_{i, x_{2n+1}}(x, \nabla_\e u) \eta X^\e_i \eta |Zu^\e|^{\beta}Zu^\e  
- (\beta + 1)
\int_{t_1} ^{t_2}\int_\Omega 
\sum_{i=1}^{2n+1}
A^\e_{i, x_{2n+1}}(x, \nabla_\e u) \eta^2 |Zu^\e|^{\beta} X^\e_i Zu^\e  
\\
& \le 2\Lambda\int_{t_1}^{t_2}
\int_\Omega (\ddelta+|\nabla_\e u^\e|^2)^{\frac{p-2}{2}}  |\nabla_\e Zu|\eta |\nabla_\e \eta||Zu^\e|^{\beta +1} +
\frac{2} {\beta + 2}\int_{t_1}^{t_2}\int_\Omega |Zu^\e|^{\beta+2} \eta \partial_t \eta 
\\
& + 2 \Lambda \int_{t_1} ^{t_2}\int_\Omega 
(\ddelta+|\nabla_\e u^\e|^2)^{\frac{p-1}{2}} \eta |\nabla_\e \eta| |Zu^\e|^{\beta+1} 
+ (\beta + 1) \Lambda 
\int_{t_1} ^{t_2}\int_\Omega 
(\ddelta+|\nabla_\e u^\e|^2)^{\frac{p-1}{2}}\eta^2 |Zu^\e|^{\beta} |\nabla_\e  Zu^\e|,  
\end{align*}
thus concluding the proof.

\end{proof}

\begin{lemma} \label{lemma3.4} Let $u^\e$ be a weak solution of \eqref{approx1} in $Q$. There exists $C_0 = C_0(n, p,\lambda, \Lambda) > 0.$ For any $t_2\ge t_1\ge 0$,  $\beta \geq 0 $ and all $\eta \in C^\infty_0(\Omega)$,
we have
\begin{align}\label{caccio2}
& \frac{1}{\beta+2}\int_\Omega \eta^2 [(\ddelta+ |\nabla_\e u^\e|^2)^{(\beta +2)/2}]\bigg|_{t_1}^{t_2} + \int_{t_1} ^{t_2} \int_\Omega
\eta^2(\ddelta + |\nabla_\e u^\e|^2)^{(p-2+\beta)/2} \sum_{i,j=1}^{2n+1} |X^\e_i X^\e_j u^\e|^2 
\\
& \leq C_0
\int_{t_1} ^{t_2} \int_\Omega(\eta^2 + |\nabla_\e\eta|^2 + \eta|Z\eta|)
(\ddelta + |\nabla_\e u^\e|^2)^{(p+\beta)/2} 
\notag\\
& + C_0(\beta + 1)^4
\int_{t_1} ^{t_2} \int_\Omega
\eta^2(
\ddelta + |\nabla_\e u^\e|^2)^{(p+\beta-2)/
2} |Zu^\e|^2.
\notag
\end{align}
\end{lemma}

\begin{proof} In view of Lemma \ref{diff-PDE} we know that, if  $u^\e\in C^{\infty}(Q)$ is a solution of
$\p_t u^\e= \sum_{i=1}^{2n+1} X^\e_i A^{\e}_i(x, \nabla_\e u^\e),$
then  $v^\e_\ell=X^\e_\ell u^\e$ solves \eqref{eqderivX}. If in the first term in the right-hand side of \eqref{eqderivX} we use the fact that $X^\e_\ell X^\e_ju^\e = X^\e_j X^\e_\ell u^\e + [X^\e_\ell,X^\e_j] u^\e = X^\e_j v^\e_\ell + s_\ell Z v^\e_\ell$, we find
\begin{equation}\label{weak2}
\p_t v^\e_\ell
=
\sum_{i,j=1}^{2n+1} X^\e_i\Big( A_{i, \xi_j}^\e (x, \nabla_\e u^\e) 
X^\e_j v^\e_\ell \Big)
+ s_\ell \sum_{i=1}^{2n+1} X^\e_i\Big( A_{i, \xi_{\ell+s_{\ell}n}}^\e (x, \nabla_\e u^\e) 
Z u^\e \Big)
+ \end{equation}
$$ +
\sum_{i=1}^{2n+1} X^\e_i\Big( A_{i, x_\ell}^\e (x, \nabla_\e u^\e) - \frac{s_\ell x_{\ell+s_\ell n}}{2} A_{i, x_{2n+1}}^\e (x, \nabla_\e u^\e) \Big) + s_\ell Z (A_{\ell+s_\ell n}^\e (x, \nabla_\e u^\e)). $$

Fix $\eta\in C^{\infty}_0(\Omega)$ and
let $\phi = \eta^2
(\ddelta+|\nabla_\e u^\e|^2
)^{\beta/2}X^\e_\ell u^\e$. Taking such $\phi$ as the 
test-function in the weak form of \eqref{weak2}, and integrating by parts the terms in divergence form, one has
{\allowdisplaybreaks
\begin{align*}
& \frac{1}{2} \int_{t_1}^{t_2}\int_\Om (\ddelta+|\nabla_\e u^\e|^2)^{\frac{\beta}{2}}\p_t \bigg[ X_\ell^\e u^\e\bigg]^2 \eta^2  
\\
& + \sum_{i,j=1}^{2n+1}  \int_{t_1}^{t_2}\int_\Om A^\e_{i \xi_j} (x, \nabla_\e u^\e) X^\e_j v^\e_{\ell}  X^\e_i\bigg(\eta^2
(\ddelta+|\nabla_\e u^\e|^2
)^{\beta/2}X^\e_\ell u^\e \bigg) 
\\
& =- s_\ell \sum_{i=1}^{2n+1}\int_{t_1}^{t_2}\int_\Om A^\e_{i \xi_{\ell+s_{\ell} n}} (x, \nabla_\e u^\e) Zu^\e X^\e_i \bigg(\eta^2
(\ddelta+|\nabla_\e u^\e|^2
)^{\beta/2}X^\e_\ell u^\e \bigg) 
\\
& +
\int_{t_1}^{t_2}\int_\Om \sum_{i=1}^{2n+1} X^\e_i\Big(A_{i, x_\ell}^\e (x, \nabla_\e u^\e) - \frac{s_lx_{\ell+s_{\ell} n}}{2} A_{i, x_{2n+1}}^\e (x, \nabla_\e u^\e) \Big)\eta^2
(\ddelta+|\nabla_\e u^\e|^2
)^{\beta/2}X^\e_\ell u^\e 
\\
&
+ s_{\ell}
 \int_{t_1}^{t_2}\int_\Om  s_l Z (A_{\ell+s_{\ell} n}^\e (x, \nabla_\e u^\e))\eta^2  (\ddelta+|\nabla_\e u^\e|^2)^{\beta/2} X^\e_\ell u^\e.
 \end{align*}}
The latter equation implies that for every $\ell=1,...,2n$ one has
\begin{align*}
& \frac{1}{\beta+2} \int_{t_1}^{t_2}\int_\Om \frac{1}{2} \int_{t_1}^{t_2}\int_\Om (\ddelta+|\nabla_\e u^\e|^2)^{\frac{\beta}{2}}\p_t \bigg[ X_\ell^\e u^\e\bigg]^2 \eta^2
\\
& +
  \sum_{i,j=1}^{2n+1} \int_{t_1}^{t_2}\int_\Om A^\e_{i \xi_j} (x, \nabla_\e u^\e) X^\e_j X^\e_\ell u^\e X^\e_iX^\e_\ell u^\e \  \eta^2 (\ddelta+|\nabla_\e u^\e|^2
)^{\beta/2}
\\
& + \sum_{i,j=1}^{2n+1} \frac{\beta}{2} \int_{t_1}^{t_2}\int_\Om A^\e_{i \xi_j} (x, \nabla_\e u^\e) X^\e_j X^\e_\ell u^\e X^\e_\ell u^\e X^\e_i( |\nabla_\e u^\e|^2) \  \eta^2 (\ddelta+|\nabla_\e u^\e|^2
)^{\frac{\beta-2}{2}} 
\\
& =
- \sum_{i,j=1}^{2n+1}\int_{t_1}^{t_2}\int_\Om A^\e_{i \xi_j} (x, \nabla_\e u^\e) X^\e_j
X^\e_\ell u^\e X^\e_\ell u^\e X^\e_i(\eta^2) (\ddelta+|\nabla_\e u|^2
)^{\beta/2}   
\\
& - s_\ell \sum_{i=1}^{2n+1}\int_{t_1}^{t_2}\int_\Om A^\e_{i \xi_{\ell+s_{\ell} n}}(x,\nabla_\e u^\e) Zu X^\e_i \bigg(\eta^2
(\ddelta+|\nabla_\e u^\e|^2
)^{\beta/2}X^\e_\ell u^\e \bigg)
\\
& - 
\int_{t_1}^{t_2}\int_\Om \sum_{i=1}^{2n+1}\Big( A_{i, x_{\ell}}^\e (x, \nabla_\e u^\e) - \frac{s_\ell x_{\ell+s_{\ell} n}}{2} A_{i, x_{2n+1}}^\e (x, \nabla_\e u^\e) \Big)  X^\e_i\Big(\eta^2
(\ddelta+|\nabla_\e u^\e|^2
)^{\beta/2}X^\e_\ell u^\e\Big)
\\
& +s_\ell \sum_{j=1}^{2n+1} \int_{t_1}^{t_2}\int_\Om Z (A_{\ell+s_{\ell} n}^\e (x, \nabla_\e u^\e)) \eta^2  (\ddelta+|\nabla_\e
u^\e|^2)^{\beta/2} X^\e_\ell u^\e 
=I^1_\ell+I^2_\ell+I^3_\ell + I^4_\ell.
\end{align*}
Summing over $\ell=1,...,2n$, by a simple application of the chain rule, and using the structural assumption \eqref{structure-epsilon}, we see that the left-hand side can be bounded from below by 
\begin{align*}
& \frac{1}{\beta+2} \int_{t_1}^{t_2}\int_\Om \p_t \bigg[(\ddelta+|\nabla_\e u^\e|^2)^{\frac{\beta}{2}+1}\bigg] \eta^2
\\
& + \sum_{\ell=1}^{2n}\sum_{i,j=1}^{2n+1}
   \int_{t_1}^{t_2}\int_\Om \eta^2 A^\e_{i \xi_j} (x,\nabla_\e u^\e) X^\e_j X^\e_\ell u^\e X^\e_iX^\e_\ell u^\e  (\ddelta+|\nabla_\e u^\e|^2
)^{\beta/2} 
\\
& +\sum_{\ell=1}^{2n}\sum_{i,j=1}^{2n+1} \frac{\beta}{2} \int_{t_1}^{t_2}\int_\Om \eta^2 A^\e_{i \xi_j} (x,\nabla_\e u^\e) X^\e_j X^\e_\ell u^\e X^\e_\ell u^\e X^\e_i( |\nabla_\e u^\e|^2)  (\ddelta+|\nabla_\e u^\e|^2
)^{\frac{\beta-2}{2}}
\\
& \ge \frac{1}{\beta+2} \int_{t_1}^{t_2}\int_\Om \p_t \bigg[ (\ddelta+|\nabla_\e u^\e|^2)^{\frac{\beta}{2}+1}\bigg] \eta^2  + 
\lambda \int_{t_1}^{t_2}\int_\Om \eta^2
(\ddelta+|\nabla_\e u^\e|^2)^{\frac{p-2+\beta}{2}} \sum_{i,j=1}^{2n+1} |X^\e_i X_j^\e u^\e|^2 
\\
& +  \frac{\lambda\beta}{4}\int_{t_1}^{t_2}\int_\Om
  \eta^2 (\ddelta+|\nabla_\e u^\e|^2
)^{\frac{p+\beta-4}{2}} |\nabla^\e( |\nabla_\e u^\e |^2)|^2.  
\end{align*}
Since the last term in the right-hand side of this estimate is nonnegative, we obtain from this bound
\begin{align}\label{gnam}
& \frac{1}{\beta+2} \int_\Omega [ (\ddelta+ |\nabla_\e u^\e|^2)^{\frac{\beta}{2}+1}\eta^2]\bigg|_{t_1}^{t_2}   + 
\lambda \int_{t_1}^{t_2}\int_\Om \eta^2
(\ddelta+|\nabla_\e u^\e|^2)^{\frac{p-2+\beta}{2}} \sum_{i,j=1}^{2n+1} |X^\e_i X_j^\e u^\e|^2 
\\
& \le \sum_{\ell = 1}^{2n} \left(I^1_\ell+I^2_\ell+I^3_\ell + I^4_\ell\right).
\notag
\end{align}
Next, we estimate each of the terms in the right-hand side  separately.  Recalling that from \eqref{structure-epsilon} one has $|A_{i\xi_j}(x,\eta)| = |\p_{\xi_j} A_i^\e(x,\eta)| \le C(\ddelta+|\eta|^2)^{\frac{p-2}{2}}$, one has that for any $\alpha>0$ there exists $C_\alpha>0$ depending only on $\alpha, p, n$ and the structure constants, such that
{\allowdisplaybreaks
\begin{align}\label{gnam2}
& \sum_{\ell=1}^{2n} I^1_\ell =-\sum_{\ell=1}^{2n} \sum_{i,j=1}^{2n+1}\int_{t_1}^{t_2}\int_\Om A^\e_{i \xi_j} (x, \nabla_\e u^\e) X^\e_j X^\e_\ell u^\e  X^\e_\ell u^\e X^\e_i(\eta^2)
(\ddelta+|\nabla_\e u^\e|^2
)^{\beta/2}
\\
& \le 2\sum_{i,j=1}^{2n+1}\int_{t_1}^{t_2}\int_\Om |\eta| (\ddelta+|\nabla_\e u^\e|^2)^{(p-2)/2}  |X^\e_j X^\e_i u^\e|  |\nabla_\e u^\e| |\nabla_\e \eta|(\ddelta+|\nabla_\e u^\e|^2
)^{\frac{\beta}{2}}
\notag\\
& \le \al \sum_{i,j=1}^{2n+1}\int_{t_1}^{t_2}\int_\Om \eta^2 (\ddelta+|\nabla_\e u^\e|^2)^{(p+\beta-2)/2} 
|X^\e_j X^\e_i u^\e|^2  + C_\alpha\int_{t_1}^{t_2}\int_\Om (\ddelta+|\nabla_\e u^\e|^2)^{(p+\beta)/2} |\nabla_\e \eta|^2.
\notag\end{align}}
Analogously, we find
\begin{align}\label{gnam3}
& \sum_{\ell=1}^{2n} I^2_\ell
\le \al \sum_{i,j=1}^{2n+1}\int_{t_1}^{t_2}\int_\Om \eta^2 (\ddelta+|\nabla_\e u^\e|^2)^{(p+\beta-2)/2} 
 |X^\e_i X^\e_j u^\e|^2
 \\
 & + C\int_{t_1}^{t_2}\int_\Om (\ddelta+|\nabla_\e u^\e|^2)^{(p+\beta)/2} |\nabla_\e \eta|^2 
+ C_\al (\beta +1)^2\int_{t_1}^{t_2}\int_\Om \eta^2 (\ddelta+|\nabla_\e u^\e|^2)^{(p+\beta-2)/2} |Z u^\e|^2.
\notag\end{align}
In a similar fashion, we obtain
\begin{align}\label{gnam4}
& \sum_{\ell=1}^{2n} I^3_\ell
\le \al \sum_{i,j=1}^{2n+1}\int_{t_1}^{t_2}\int_\Om (\ddelta+|\nabla_\e u^\e|^2)^{(p+\beta-2)/2} |X^\e_i X^\e_j u^\e|^2
\eta^2 
\\
&+ C_\alpha(\beta +1)^2\int_{t_1}^{t_2}\int_\Om (\ddelta+|\nabla_\e u^\e|^2)^{(p+\beta)/2} (|\nabla_\e \eta|^2 +|\eta|^2).
\notag
\end{align}
Finally, integrating by parts twice, and using the structural assumptions, one has
{\allowdisplaybreaks
\begin{align}\label{gnam5}
& \sum_{\ell=1}^{2n} I^4_\ell =
 -\sum_{\ell=1}^{2n} \int_{t_1}^{t_2}\int_\Om Z (A_{\ell+s_{\ell}n} (x, \nabla_\e u^\e) ) \eta^2 (\ddelta+|\nabla_\e u^\e|^2)^{\beta/2} X^\e_\ell u^\e
 \\
& = 2\sum_{\ell=1}^{2n} \int_{t_1}^{t_2}\int_\Om  A_{\ell+s_{\ell}n} (x, \nabla_\e u^\e) \eta Z\eta (\ddelta+|\nabla_\e u^\e|^2)^{\beta/2} X^\e_\ell u^\e
\notag\\
& +\beta \sum_{\ell=1}^{2n} \sum_{j=1}^{2n+1}\int_{t_1}^{t_2}\int_\Om  A_{\ell+s_{\ell}n} (x, \nabla_\e u^\e) \eta^2 (\ddelta+|\nabla_\e u^\e|^2)^{\frac{\beta-2}{2}} X_j u^\e X_j Zu^\e X^\e_\ell u^\e
\notag\\
& +\sum_{\ell=1}^{2n} \int_{t_1}^{t_2}\int_\Om  A_{\ell+s_{\ell}n} (x, \nabla_\e u^\e) \eta^2 (\ddelta+|\nabla_\e u^\e|^2)^{\beta/2} X^\e_\ell Zu^\e 
\notag\\
& = 2\sum_{\ell=1}^{2n} \int_{t_1}^{t_2}\int_\Om  A_{\ell+s_{\ell}n} (x, \nabla_\e u^\e) \eta Z\eta (\ddelta+|\nabla_\e u^\e|^2)^{\beta/2} X^\e_\ell u^\e 
\notag\\
& -\beta \sum_{\ell=1}^{2n} \sum_{j=1}^{2n+1}\int_{t_1}^{t_2}\int_\Om  X_j \bigg(A_{\ell+s_{\ell}n} (x, \nabla_\e u^\e) \eta^2 (\ddelta+|\nabla_\e u^\e|^2)^{\frac{\beta-2}{2}} X_j u^\e X^\e_\ell u^\e\bigg) Zu^\e
\notag\\
& -\sum_{\ell=1}^{2n} \int_{t_1}^{t_2}\int_\Om  X^\e_\ell\bigg( A_{\ell+s_{\ell}n} (x, \nabla_\e u^\e) \eta^2 (\ddelta+|\nabla_\e u^\e|^2)^{\beta/2} \bigg)Zu^\e 
\notag\\
& \le \al \sum_{i,j=1}^{2n+1}\int_{t_1}^{t_2}\int_\Om (\ddelta+|\nabla_\e u^\e|^2)^{(p+\beta-2)/2} 
 |X^\e_i X^\e_j u^\e|^2
 \eta^2
 \notag\\
 & + C(\beta + 1)\int_{t_1}^{t_2}\int_\Om (\ddelta+|\nabla_\e u^\e|^2)^{(p+\beta)/2}\Big(\eta^2 + |\nabla_\e \eta|^2 + |\eta Z\eta|\Big) 
\notag\\
& + C_\alpha (\beta + 1)^4\int_{t_1}^{t_2}\int_\Om (\ddelta+|\nabla_\e u^\e|^2)^{(p+\beta-2)/2} |Z u^\e|^2.
\notag\end{align}
}
Combining \eqref{gnam2}-\eqref{gnam5} with \eqref{gnam}, we reach the desired conclusion \eqref{caccio2}.

\end{proof}

In the case $\beta=0$ we obtain the following stronger estimate, which we will need in the sequel. We denote by $||\cdot||$ the $L^\infty$ norm of a function on the parabolic cylinder $Q$.

\begin{lemma} \label{lemma3.4} Let $u^\e$ be a weak solution of \eqref{approx1} in $Q$, let $t_2\ge t_1\ge 0$,   and $\eta\in C^{1}([0,T], C^\infty_0 (\Om))$  be such that $0\le \eta \le1$, and for which
$ ||\partial_t\eta||\leq C ||\nabla_\e\eta||^2$, where $C>0$ is a universal constant.
 For every $\alpha>0$ there exists $C_\alpha>0$ such that 
\begin{align*}
& \frac{1}{2}\int_\Omega ( (\ddelta+ |\nabla_\e u^\e|^2)\eta^2)\Big|_{t_1}^{t_2}   + 
\lambda\int_{t_1}^{t_2}\int_\Om
(\ddelta+|\nabla_\e u^\e|^2)^{\frac{p-2 }2} \sum_{i,j=1}^{2n+1} |X^\e_i X_j^\e u^\e|^2 \eta^2
\\
& \leq 
 \alpha   
\int_{t_1}^{t_2}\int_\Omega |Zu^\e|^{2} \eta^3
 + C_\alpha\int_{t_1}^{t_2}\int_\Om (\ddelta+|\nabla_\e u^\e|^2)^{p/2}\Big(\eta^2 + |\nabla_\e \eta|^2 + |\eta Z\eta|\Big).
 \end{align*}
\end{lemma}

\begin{proof} In view of Lemma \ref{diff-PDE} we notice that, if  $u^\e\in C^{\infty}(Q)$ is a solution of
$\p_t u^\e= \sum_{i=1}^{2n+1} X^\e_i A^{\e}_i(x, \nabla_\e u^\e)$,
then  $v^\e_\ell=X^\e_\ell u^\e$ solves
\begin{align}\label{weak2}
 \p_t v^\e_\ell
& =
\sum_{i,j=1}^{2n+1} X^\e_i\Big( X_\ell^\e(A_{i}^\e (x, \nabla_\e u^\e) )
 \Big)
 +
\sum_{i=1}^{2n+1} X^\e_i\Big( A_{i, x_\ell}^\e (x, \nabla_\e u^\e) - \frac{s_{\ell} x_{\ell+s_{\ell}n}}{2} A_{i, x_{2n+1}}^\e (x, \nabla_\e u^\e) \Big) 
\\
&+ s_\ell Z (A_{\ell+s_{\ell}n}^\e (x, \nabla_\e u^\e)).
\notag
\end{align}
With $\eta$ as in the statement of the lemma, we take  
$\phi = \eta^2
X^\e_\ell u^\e$ as a 
test function in the weak form of \eqref{weak2}. Integrating by parts  the terms in divergence form,  one has
\begin{align*}
& \frac{1}{2} \int_{t_1}^{t_2}\int_\Om \eta^2 \p_t  (X_\ell^\e  u^\e)^2   + 
\sum_{i=1}^{2n+1}  \int_{t_1}^{t_2}\int_\Om X_\ell (A^\e_{i } (x, \nabla_\e u^\e) ) X^\e_i\bigg(\eta^2
X^\e_\ell u^\e \bigg) 
\\
& =
\int_{t_1}^{t_2}\int_\Om \eta^2 \sum_{i=1}^{2n+1} X^\e_i\Big(A_{i, x_\ell}^\e (x,\nabla_\e u^\e) - \frac{s_\ell x_{\ell+s_{\ell}n}}{2} A_{i, x_{2n+1}}^\e (x, \nabla_\e u^\e) \Big)
X^\e_\ell u^\e
\\
& + s_\ell
 \int_{t_1}^{t_2}\int_\Om  \eta^2  Z (A_{\ell+s_{\ell}n}^\e (x, \nabla_\e u^\e))  X^\e_{\ell} u^\e.
\end{align*}
The gives 
\begin{align*}
& \frac{1}{ 2} \int_{t_1}^{t_2}\int_\Om \eta^2 \p_t (X_\ell^\e u^\e)^2   +\sum_{i,j=1}^{2n+1}  \int_{t_1}^{t_2}\int_\Om \eta^2 A^\e_{i, \xi_j } (x, \nabla_\e u^\e) X^\e_\ell X^\e_j u^\e  
X^\e_\ell X^\e_i u^\e 
\\
& = - \sum_{i,j=1}^{2n+1}  \int_{t_1}^{t_2}\int_\Om \eta^2 X_\ell (A^\e_{i } (x, \nabla_\e u^\e) )  
Z u^\e - \sum_{i,j=1}^{2n+1}  \int_{t_1}^{t_2}\int_\Om X_\ell (A^\e_{i } (x, \nabla_\e u^\e) )  \eta X^\e_i\eta X^\e_\ell u^\e
\\
& - \int_{t_1}^{t_2}\int_\Om \sum_{i=1}^{2n+1}\Big( A_{i, x_\ell}^\e (x, \nabla_\e u^\e) - \frac{s_\ell x_{\ell+s_{\ell}n}}{2} A_{i, x_{2n+1}}^\e (x,\nabla_\e u^\e) \Big)  X^\e_i\Big( \eta^2
X^\e_\ell u^\e \Big)
\\
& + s_\ell \sum_{j=1}^{2n+1} \int_{t_1}^{t_2}\int_\Om \eta^2  Z (A_{\ell+s_{\ell}n}^\e (x,\nabla_\e u^\e))   X^\e_\ell u^\e
=I^1_\ell+I^2_\ell+I^3_\ell + I^4_\ell.
\end{align*}
Summing over $\ell=1,...,2n$, in view of the structural hypothesis \eqref{structure-epsilon}, after an integration by parts in the first term in the left-hand side we obtain the following bound
\begin{align*}
& \frac 12 \int_\Om   (\ddelta+|\nabla_\e u^\e|^2) \eta^2 \Big|^{t_2}_{t_1} + 
\lambda\int_{t_1}^{t_2}\int_\Om
(\ddelta+|\nabla_\e u^\e|^2)^{\frac{p-2 }2} \sum_{i,j=1}^{2n+1} |X^\e_i X_j^\e u^\e|^2 \eta^2
\\
& \leq I^1_\ell+I^2_\ell+I^3_\ell + I^4_\ell + \int_{t_1}^{t_2}\int_\Om (\ddelta+|\nabla_\e u^\e|^2) \eta\partial_t\eta.
\end{align*}
Next, we estimate each of the terms in the right-hand side  separately.  Recalling that $|A^\e_{i\xi_j}(x,\eta)|\le C(\ddelta+|\eta|^2)^{\frac{p-2}{2}},$ we find that for any $\alpha_1, \alpha_2 >0$ there exist $C_{\alpha_1}, C_{\alpha_2} >0$, depending only on $\alpha_1, \alpha_2, p, n$ and the structure constants, such that
{\allowdisplaybreaks
\begin{align*}
& \sum_{\ell=1}^{2n} I^1_\ell=\sum_{\ell=1}^{2n} 
\sum_{i,j=1}^{2n+1}  \int_{t_1}^{t_2}\int_\Om \eta^2 X_\ell (A^\e_{i } (x, \nabla_\e u^\e))  
Z u^\e 
\\
&=-2\sum_{\ell=1}^{2n} 
\sum_{i=1}^{2n+1}  \int_{t_1}^{t_2}\int_\Om A^\e_{i } (x, \nabla_\e u^\e)   \eta X_\ell \eta
Z u^\e 
-\sum_{\ell=1}^{2n} 
\sum_{i=1}^{2n+1}  \int_{t_1}^{t_2}\int_\Om \eta^2 A^\e_{i }(x, \nabla_\e u^\e)   
X^\e_\ell Z u^\e 
\\
& =-2\sum_{\ell=1}^{2n} 
\sum_{i=1}^{2n+1}  \int_{t_1}^{t_2}\int_\Om A^\e_{i } (x, \nabla_\e u^\e)  \eta X_\ell \eta
Z u^\e 
\\
&+\sum_{\ell=1}^{2n} 
\sum_{i=1}^{2n+1}  \int_{t_1}^{t_2}\int_\Om A^\e_{i } (x, \nabla_\e u^\e)   2\eta Z\eta 
X^\e_\ell u^\e 
+\sum_{\ell=1}^{2n} 
\sum_{i,j=1}^{2n+1}  \int_{t_1}^{t_2}\int_\Om \eta^2 A^\e_{i \xi_j} (x, \nabla_\e u^\e) X_j^\e Z u^\e   
X^\e_\ell u^\e 
\\ 
& \leq \int_{t_1}^{t_2}\int_\Om (\ddelta + |\nabla_\e u^\e|^2)^{p-1} \eta |\nabla_\e \eta|
|Z u^\e| +
\sum_{\ell=1}^{2n} 
\sum_{i=1}^{2n+1}  \int_{t_1}^{t_2}\int_\Om (\ddelta + |\nabla_\e u^\e|^2)^{p}  |2\eta Z\eta |
\\
& +\sum_{\ell=1}^{2n} 
\sum_{i,j=1}^{2n+1}  \int_{t_1}^{t_2}\int_\Om \eta^2 (\ddelta + |\nabla_\e u^\e|^2)^{(p-1)/2} |\nabla_\e Z u^\e|   
 \\
 & \leq 
\alpha_1\int_{t_1}^{t_2}\int_\Om  (\ddelta + |\nabla_\e u^\e|^2)^{(p-2)/2} \eta^2
|Z u^\e|^2 +
C_{\alpha_1}\int_{t_1}^{t_2}\int_\Om (\ddelta + |\nabla_\e u^\e|^2)^{p/2} |\nabla_\e\eta|^2
\\
& +
\sum_{\ell=1}^{2n} 
\sum_{i=1}^{2n+1}  \int_{t_1}^{t_2}\int_\Om (\ddelta + |\nabla_\e u^\e|^2)^{p/2}  |2\eta Z\eta |
\\
& +\frac{\alpha_2}{||\nabla_\e\eta||^2}\sum_{\ell=1}^{2n} 
\sum_{i,j=1}^{2n+1}  \int_{t_1}^{t_2}\int_\Om \eta^4 (\ddelta + |\nabla_\e u^\e|^2)^{(p-2)/2} |\nabla_\e Z u^\e|^2   +
C_{\alpha_2}||\nabla_\e\eta||^2\int_{t_1}^{t_2}\int_{supp(\eta)} (\ddelta + |\nabla_\e u^\e|^2)^{p/2}.
\end{align*}
}
Now, we apply Lemma \ref{stimaZu} to find, for any $\alpha>0$,
\begin{align*}
& \frac{\alpha}{||\nabla_\e \eta||^2}\sum_{\ell=1}^{2n} 
\sum_{i,j=1}^{2n+1}  \int_{t_1}^{t_2}\int_\Om (\ddelta + |\nabla_\e u^\e|^2)^{(p-2)/2} |\nabla_\e Z u^\e|^2   \eta^4
\\
& \le \alpha C  
\int_{t_1} ^{t_2}\int_\Omega
(\ddelta+|\nabla_\e u^\e|^2)^{\frac{p-2}{2}} |Zu^\e|^{2}\eta^2
+ \frac{\alpha  ||\partial_t\eta||}{
||\nabla_\e\eta||^2} 
\int_{t_1}^{t_2}\int_\Omega |Zu^\e|^{2} \eta^3 
\\
&+ \alpha 
\int_{t_1} ^{t_2}\int_\Omega 
(\ddelta+|\nabla_\e u^\e|^2)^{\frac{p}{2}} \eta^4. 
\end{align*}
Analogously, 
\begin{align*}
& \sum_{\ell=1}^{2n} I^2_\ell + \sum_{\ell=1}^{2n} I^3_\ell
\le \al \sum_{i,j=1}^{2n+1}\int_{t_1}^{t_2}\int_\Om (\ddelta+|\nabla_\e u^\e|^2)^{(p-2)/2} 
 |X^\e_i X^\e_j u^\e|^2
 \eta^2
 \\
&  + C\int_{t_1}^{t_2}\int_\Om (\ddelta+|\nabla_\e u^\e|^2)^{p/2} |\nabla_\e \eta|^2.
\end{align*}
Using the structure conditions, one has
\begin{align*}
& \sum_{\ell=1}^{2n} I^4_\ell 
\le \sum_{i,j=1}^{2n+1}\int_{t_1}^{t_2}\int_\Om (\ddelta+|\nabla_\e u^\e|^2)^{(p-1)/2} 
 |\nabla_\e Zu^\e|
\eta^2
\\
& \le \al\sum_{i,j=1}^{2n+1}\int_{t_1}^{t_2}\int_\Om (\ddelta+|\nabla_\e u^\e|^2)^{(p-2)/2} 
 |\nabla_\e Z u^\e|^2
 \eta^2
 \\
 & + C_\alpha\int_{t_1}^{t_2}\int_\Om (\ddelta+|\nabla_\e u^\e|^2)^{p/2}\Big(\eta^2 + |\nabla_\e \eta|^2 + |\eta Z\eta|\Big),
 \end{align*}
thus concluding the proof.

\end{proof}

Next, we need to establish mixed type Caccioppoli inequalities, where the left-hand side includes terms with both horizontal derivatives and derivatives along the second layer of the stratified Lie algebra of $\Hn$.

\begin{lemma} \label{cacciopoXZ}Set $T >t_2 > t_1> 0$. Let $u^\e$ be a weak solution of \eqref{approx1} in $Q=\Om\times (0,T)$. Let 
$\beta\ge 2$ and let $\eta\in C^1((0,T), C^{\infty}_0(\Om))$,  with $0\le \eta \leq 1$. For all $\alpha \leq 1$ there exist  constants
$C_\Lambda$, $C_\alpha = C(\alpha, \lambda, \Lambda)>0$ such that  
{\allowdisplaybreaks
\begin{align}\label{longwaydown}
& \int_{t_1}^{t_2} \int_\Om \eta^{\beta+2}(\ddelta+|\nabla_\e u^\e|^2)^{\frac{p-2}{2}} |Zu^\e|^\beta \sum_{i,j=1}^{2n+1} |X^\e_iX^\e_j u^\e|^2 
\\
& + \int_\Om
 \eta^{\beta+2}|Zu^\e|^\beta |\nabla_\e u^\e|^2\bigg|_{t_1}^{t_2}
\notag\\
&  +  (\beta+ 1)^2\int_{t_1}^{t_2} \int_\Om (\ddelta+|\nabla_\e u^\e|^2)^{\frac{p-2}{2}} |\nabla_\e Zu^\e|^2   |Zu^\e|^{\beta-2} \eta^{\beta+2} 
|\nabla_\e u^\e|^2 
\notag\\
& \leq C_\al (\beta+1)^2(1+ |\nabla_\e \eta||_{L^\infty}^2)\int_{t_1}^{t_2} \int_\Om (\eta^{\beta} +  \eta^{\beta+4}) (\ddelta+|\nabla_\e u^\e|^2)^{\frac{p}{2}} |Zu^\e|^{\beta-2} \sum_{i,j=1}^{2n+1} |X^\e_iX^\e_j u^\e|^2
\notag\\
& + \frac{2\al}{(1 + ||\nabla_\e \eta||^2)(\beta + 2)}\int_{t_1}^{t_2}\int_\Omega |Zu^\e|^{\beta+2} 
\eta^{\beta+3} |\partial_t \eta| dx + \frac{\al}{(\beta + 2)^2}\int_\Omega |Zu^\e|^{\beta+2} \eta^{\beta+4}\Big|_{t=t_1}
\notag\\
&+ C_{\Lambda}(\beta+1)^2
\int_{t_1}^{t_2}  \int_\Om
(\ddelta+|\nabla_\e u^\e|^2)^{\frac{p+2}{2}} |Zu^\e|^{\beta-2}    \eta^{\beta+2} 
\notag\\
& +
\int_{t_1}^{t_2} \int_{\Om} |Zu^\e|^{\beta}  |\nabla_\e u^\e|^2  \p_t ( \eta^{\beta+2}).
\notag
\end{align}
}
\end{lemma}

\begin{proof}
Let $\eta \in C^\infty_0 (\Omega\times (0,T))$
 be a nonnegative cutoff function. Fix $\beta\geq 2 $ and $\ell \in
\{1, . . . , 2n\}.$ Note that
$$\partial_t(|X^\e_\ell u^\e|^2 |Zu^\e|^\beta)= 2  X^\e_\ell u^\e \partial_ tX^\e_\ell u^\e|Zu^\e|^\beta +\beta|X^\e_\ell u^\e|^2  |Zu^\e|^{\beta-2} Zu^\e \partial_t Zu^\e,$$
which suggests to use
$ 2  X^\e_\ell u^\e |Zu^\e|^\beta$ as a test function in the equation \eqref{eqderivX} satisfied by $X^\e_\ell u^\e $ and to choose
$\beta|X^\e_\ell u^\e|^2  |Zu^\e|^{\beta-2} Zu^\e $ as a test function in the equation \eqref{eqZ} satisfied by $Zu^\e$. 
Equation \eqref{eqderivX} becomes in weak form
\begin{align*} 
& \int_{t_1}^{t_2} \int_\Om \p_t X^\e_\ell u^\e  \phi  = 
 - \sum_{i,j=1}^{2n+1}  \int_{t_1}^{t_2} \int_\Om 
\Big( A_{i, \xi_j}^\e (x, \nabla_\e u^\e) X^\e_\ell X^\e_ju^\e \Big)
X^\e_i \phi + s_\ell Z (A_{\ell+s_{\ell}n}^\e (x, \nabla_\e u^\e)) \phi
\\
&
-  \sum_{i=1}^{2n+1}\int_{t_1}^{t_2} \int_\Om 
\Big( A_{i, x_\ell}^\e (x, \nabla_\e u^\e) - 
\frac{s_\ell x_{\ell+s_{\ell}n}}{2} A_{i, x_{2n+1}}^\e (x, \nabla_\e u^\e) \Big)
X^\e_i \phi. 
\end{align*}
Consequently,
if we substitute the test function $\phi=2 \eta^{\beta+2}|Zu|^\beta X^\e_\ell u$, we obtain
{\allowdisplaybreaks
\begin{align}\label{tXestimate} 
& 2\int_{t_1}^{t_2} \int_\Om \p_t X^\e_\ell u^\e   \eta^{\beta+2}|Zu^\e|^\beta X^\e_\ell u^\e
\\
& + 
2  \sum_{i,j=1}^{2n+1} \int_{t_1}^{t_2} \int_\Om  A_{i, \xi_j}^\e (x, \nabla_\e u^\e) X^\e_\ell X^\e_ju^\e  
\eta^{\beta+2}|Zu^\e|^\beta X^\e_\ell  X^\e_i u^\e 
\notag\\
& = - \sum_{i,j=1}^{2n+1}2\int_{t_1}^{t_2} \int_\Om A_{i, \xi_j}^\e (x, \nabla_\e u^\e) X^\e_\ell X^\e_ju^\e   X^\e_i \bigg(\eta^{\beta+2}|Zu^\e|^\beta\bigg) X^\e_\ell u^\e 
\notag\\
& -  2 \sum_{i,j=1}^{2n+1}\int_{t_1}^{t_2} \int_\Om  A_{i, \xi_j}^\e (x, \nabla_\e u^\e) X^\e_\ell X^\e_ju^\e  \eta^{\beta+2}|Zu^\e|^\beta[X^\e_i, X^\e_\ell] u^\e 
\notag\\
& - 2 s_l\int_{t_1}^{t_2} \int_\Om Z(A_{\ell+s_{\ell}n}(x, \nabla_\e u^\e) )\eta^{\beta+2}|Zu^\e|^\beta X^\e_\ell u^\e  
\notag\\
& -2\sum_{i=1}^{2n+1}\int_{t_1}^{t_2} \int_\Om 
\Big( A_{i, x_\ell}^\e (x, \nabla_\e u^\e) - 
\frac{s_\ell x_{\ell+s_{\ell}n}}{2} A_{i, x_{2n+1}}^\e (x, \nabla_\e u^\e) \Big)
X^\e_i \Big(
\eta^{\beta+2}|Zu^\e|^\beta X^\e_\ell u^\e\Big)
\notag\\
& = I_\ell^1 + I_\ell^2  + I_\ell^3 + I_\ell^4.
\notag
\end{align}}
We will show that these terms satisfy the following estimate
{\allowdisplaybreaks
\begin{align}\label{stimaI1}
& \sum_{k=1}^4\sum_{\ell=1}^{2n} |I_\ell^k|\leq 
\al\int_{t_1}^{t_2} \int_\Om \eta^{\beta+2}(\ddelta+|\nabla_\e u^\e|^2)^{\frac{p-2}{2}} |Zu^\e|^\beta \sum_{i,j=1}^{2n+1} |X^\e_iX^\e_j u^\e|^2  
\\
& + C_\al (\beta+1)^2
(1 + ||\nabla_\e \eta||^2_{L^\infty}
\int_{t_1}^{t_2} \int_\Om (\eta^{\beta} + \eta^{\beta+4}) (\ddelta+|\nabla_\e u^\e|^2)^{\frac{p}{2}} |Zu^\e|^{\beta-2} \sum_{i,j=1}^{2n+1} |X^\e_iX^\e_j u^\e|^2  
\notag\\
& + \frac{2\al}{(1 + ||\nabla_\e \eta||^2)(\beta + 2)}\int_{t_1}^{t_2}\int_\Omega |Zu^\e|^{\beta+2} 
\eta^{\beta+3} |\partial_t \eta|  + \frac{\al}{(\beta + 2)^2}\int_\Omega |Zu^\e|^{\beta+2} \eta^{\beta+4}  \Big|_{t=t_1}
\notag\\
& +\al (\beta+1)^2\int_{t_1}^{t_2} \int_\Om  (\ddelta+|\nabla_\e u^\e|^2)^{\frac{p-2}{2}} 
\eta^{\beta+4} |Zu^\e|^{\beta-2}|\nabla_\e Zu^\e|^2  |\nabla_\e u^\e|^2.
\notag
\end{align}} 
We first note that 
{\allowdisplaybreaks
\begin{align*}
& \sum_{\ell=1}^{2n}|I_\ell^1| \leq 2 \sum_{\ell=1}^{2n}\sum_{i,j=1}^{2n+1} \int_{t_1}^{t_2} \int_\Om  |A_{i, \xi_j}^\e (x, \nabla_\e u^\e) X^\e_\ell X^\e_ju^\e  
X_i \Big(\eta^{\beta+2}|Zu^\e|^\beta   \Big) X^\e_\ell u^\e| 
\\
& \leq 2n \Lambda(\beta+2) \sum_{\ell=1}^{2n}\sum_{j=1}^{2n+1}
\int_{t_1}^{t_2} \int_\Om  (\ddelta+|\nabla_\e u^\e|^2)^{\frac{p-1}{2}}  |X^\e_\ell X^\e_ju^\e | 
\eta^{\beta+1} |\nabla_\e \eta|  |\nabla_\e u^\e| |Zu^\e|^\beta
\\
& +
2n \beta  \sum_{\ell=1}^{2n}\sum_{j=1}^{2n+1}\int_{t_1}^{t_2} \int_\Om  (\ddelta+|\nabla_\e u^\e|^2)^{\frac{p-1}{2}}  |X^\e_\ell X^\e_ju^\e | 
\eta^{\beta+2} |Zu^\e|^{\beta-1} |\nabla_\e Zu^\e|  
\\
& \leq \al\int_{t_1}^{t_2} \int_\Om \eta^{\beta+2}(\ddelta+|\nabla_\e u^\e|^2)^{\frac{p-2}{2}} |Zu^\e|^\beta \sum_{i,j=1}^{2n+1} |X^\e_iX^\e_j u^\e|^2 
\\
& + C_\al (\beta+1)^2\int_{t_1}^{t_2} \int_\Om \eta^{\beta}|\nabla_\e \eta|^2 (\ddelta+|\nabla_\e u^\e|^2)^{\frac{p}{2}} |Zu^\e|^\beta  
\\
& + C_\al (\beta+1)^2(1 + ||\nabla_\e \eta||^2)\int_{t_1}^{t_2} \int_\Om \eta^{\beta} (\ddelta+|\nabla_\e u^\e|^2)^{\frac{p}{2}} |Zu^\e|^{\beta-2} \sum_{i,j=1}^{2n+1} |X^\e_iX^\e_j u^\e|^2 
\\
& +\frac{\al}{1 + ||\nabla_\e \eta||^2}\int_{t_1}^{t_2} \int_\Om \eta^{\beta+4} (\ddelta+|\nabla_\e u^\e|^2)^{\frac{p-2}{2}}
 |Zu^\e|^{\beta} |\nabla_\e Z u^\e|^2.
 \end{align*}}
The last term can be estimated, as follows, using Lemma \ref{stimaZu}:
{\allowdisplaybreaks
\begin{align}\label{estimateXZu} 
& \al\int_{t_1}^{t_2} \int_\Om \eta^{\beta+4} (\ddelta+|\nabla_\e u^\e|^2)^{\frac{p-2}{2}}
 |Zu^\e|^{\beta} |\nabla_\e Z u^\e|^2
 \\
 & \le \al C_{\Lambda, \lambda}
\int_{t_1} ^{t_2}\int_\Omega
(\ddelta+|\nabla_\e u^\e|^2)^{\frac{p-2}{2}}|\nabla_\e\eta|^2 \eta^{\beta+2}|Zu^\e|^{\beta+2}
\notag\\
& + \frac{2\al } {\beta + 2}\int_{t_1}^{t_2}\int_\Omega |Zu^\e|^{\beta+2} \eta^{\beta+3} \partial_t \eta
\notag\\
& + \frac{\al}{(\beta + 1)^2}\int_\Omega |Zu^\e|^{\beta+2} \eta^{\beta+4}  \Big|_{t=t_1}
+ \al C_{\Lambda, \lambda}\int_{t_1} ^{t_2}\int_\Omega 
(\ddelta+|\nabla_\e u^\e|^2)^{\frac{p}{2}} \eta^{\beta +4}  |Zu^\e|^{\beta} \notag\\
& \le\al C_{\Lambda, \lambda}
\int_{t_1} ^{t_2}\int_\Omega
(\ddelta+|\nabla_\e u^\e|^2)^{\frac{p-2}{2}}|\nabla_\e\eta|^2 \eta^{\beta+2}|Zu^\e|^{\beta}
\sum_{ij}|X_iX_j u|^2
\notag\\
& + \frac{2\al } {\beta + 2}\int_{t_1}^{t_2}\int_\Omega |Zu^\e|^{\beta+2} \eta^{\beta+3} \partial_t \eta  
 + \frac{\al}{(\beta + 1)^2}\int_\Omega |Zu^\e|^{\beta+2} \eta^{\beta+4}  \Big|_{t=t_1}
\notag\\
& +
\al C_{\Lambda, \lambda}\int_{t_1} ^{t_2}\int_\Omega 
(\ddelta+|\nabla_\e u^\e|^2)^{\frac{p}{2}} \eta^{\beta +4}  |Zu^\e|^{\beta}.
\notag
\end{align}
}
From here estimate \eqref{stimaI1} holds. Integrating by parts we have  
{\allowdisplaybreaks
\begin{align*}
& \sum_{\ell=1}^{2n}|I_\ell^2| = - 2\sum_{\ell=1}^{2n}\sum_{i,j=1}^{2n+1} \int_{t_1}^{t_2} \int_\Om  A_{i}^\e (x, \nabla_\e u^\e) X^\e_l\Big( \eta^{\beta+2}|Zu^\e|^\beta[X^\e_i, X^\e_\ell] u^\e\Big) 
\\
& \leq  2(\beta+2)\int_{t_1}^{t_2} \int_\Om (\ddelta+|\nabla_\e u^\e|^2)^{\frac{p-1}{2}} \eta^{\beta+1} |\nabla_\e\eta| |Zu^\e|^{\beta+1}  +
\\
&+  2(\beta+1)\int_{t_1}^{t_2} \int_\Om  (\ddelta+|\nabla_\e u^\e|^2)^{\frac{p-1}{2}} \eta^{\beta+2}  |Zu^\e|^{\beta} |\nabla_\e Zu^\e| 
\\
&\leq  C_\alpha(\beta+1)^2\int_{t_1}^{t_2} \int_\Om (\ddelta+|\nabla_\e u^\e|^2)^{\frac{p}{2}} \eta^{\beta}  |Zu^\e|^{\beta-2} \sum_{i,j=1}^{2n+1} |X^\e_iX^\e_j u^\e|^2 
\\
& + \alpha\int_{t_1}^{t_2} \int_\Om (\ddelta+|\nabla_\e u^\e|^2)^{\frac{p-2}{2}} \eta^{\beta+2}  |Zu^\e|^{\beta } \sum_{i,j=1}^{2n+1} |X^\e_iX^\e_j u^\e|^2 
\\
& +  \al\int_{t_1}^{t_2} \int_\Om  (\ddelta+|\nabla_\e u^\e|^2)^{\frac{p-2}{2}} \eta^{\beta+4}  |Zu^\e|^{\beta} |\nabla_\e Zu^\e|^2  
\\
&+  C_\al(\beta+1)^2\int_{t_1}^{t_2} \int_\Om  (\ddelta+|\nabla_\e u^\e|^2)^{\frac{p}{2}} \eta^{\beta}  |Zu^\e|^{\beta-2} \sum_{i,j=1}^{2n+1} |X^\e_iX^\e_j u^\e|^2.
\end{align*}}
From here, using inequality \eqref{estimateXZu}, we deduce that $I^2_\ell$ satisfies inequality 
\eqref{stimaI1}. 
The estimate of $I_\ell^3$ can be made as follows:
{\allowdisplaybreaks
\begin{align*}
& |I_\ell^3|\leq 
 \al\int_{t_1}^{t_2} \int_\Om  (\ddelta+|\nabla_\e u^\e|^2)^{\frac{p-2}{2}} \eta^{\beta+4}  |Zu^\e|^{\beta} |\nabla_\e Zu^\e|^2 
\\
&+  C_\al\int_{t_1}^{t_2} \int_\Om  (\ddelta+|\nabla_\e u^\e|^2)^{\frac{p}{2}} \eta^{\beta}  |Zu^\e|^{\beta-2} \sum_{i,j=1}^{2n+1} |X^\e_iX^\e_j u^\e|^2
\\
& + \Lambda \int_{t_1}^{t_2} \int_\Om  (\ddelta+|\nabla_\e u^\e|^2)^{\frac{p}{2}} \eta^{\beta+2}  |Zu^\e|^{\beta}. 
\end{align*}}
From here and \eqref{estimateXZu} the inequality \eqref{stimaI1} follows. The estimate of $I^4_\ell$ is analogous:
{\allowdisplaybreaks
\begin{align*}
& |I_\ell^4|\leq 2(\beta+1)\Lambda \int_{t_1}^{t_2} \int_\Om  (\ddelta+|\nabla_\e u^\e|^2)^{\frac{p-1}{2}} 
\eta^{\beta+1} |\nabla_\e\eta| |Zu^\e|^{\beta}|\nabla_\e u^\e|  
\\
& + 2(\beta+1)\Lambda \int_{t_1}^{t_2} \int_\Om  (\ddelta+|\nabla_\e u^\e|^2)^{\frac{p-1}{2}} 
\eta^{\beta+2} |Zu^\e|^{\beta-1}  |\nabla_\e Zu^\e| |\nabla_\e u^\e| 
\\
& +  \Lambda \sum_{i,j=1}^{2n+1}\int_{t_1}^{t_2} \int_\Om  (\ddelta+|\nabla_\e u^\e|^2)^{\frac{p-1}{2}} 
\eta^{\beta+2} |Zu^\e|^{\beta}   |X^\e_i X^\e_ju^\e| 
\\
& \leq \al (\beta+1)^2\int_{t_1}^{t_2} \int_\Om  (\ddelta+|\nabla_\e u^\e|^2)^{\frac{p-2}{2}} 
\eta^{\beta+4} |Zu^\e|^{\beta-2}|\nabla_\e Zu^\e|^2  |\nabla_\e u^\e|^2 
\\
& +\al \sum_{i,j=1}^{2n+1}\int_{t_1}^{t_2} \int_\Om  (\ddelta+|\nabla_\e u^\e|^2)^{\frac{p-2}{2}} 
\eta^{\beta+2} |Zu^\e|^{\beta} |X^\e_i X^\e_ju^\e|^2  
\\
& + C_\al(\beta+1)(1 + ||\nabla_\e\eta||_{L^\infty}^2) \int_{t_1}^{t_2} \int_\Om  (\ddelta+|\nabla_\e u^\e|^2)^{\frac{p}{2}} 
 |Zu^\e|^{\beta} (\eta^{\beta} + \eta^{\beta+2} ).
\end{align*}}
We now recall the following pde from Lemma \ref{eqZ}   
$$\partial_t Zu^\e=
\sum_{i,j=1}^{2n+1}
X^\e_i(A^\e_{i, \xi_j}(x, \nabla_\e u^\e)X^\e_j Zu^\e) +
\sum_{i=1}^{2n+1}
X^\e_i(A^\e_{i, x_{2n+1}}(x, \nabla_\e u^\e)). 
$$
Substituting in this equation the test function $\phi=\beta |Zu|^{\beta-2}Zu \eta^{\beta+2} |\nabla_\e u^\e|^2$, one obtains
{\allowdisplaybreaks
\begin{align}\label{tZestimate}
& \beta\int_{t_1}^{t_2} \int_\Om \p_tZu^\e  |Zu^\e|^{\beta-2}Zu^\e \eta^{\beta+2} |\nabla_\e u^\e|^2 
\\
& + 
\beta(\beta-1) \sum_{i,j=1}^{2n+1} \int_{t_1}^{t_2} \int_\Om A_{i\xi_j}(x,\nabla_\e u^\e) X^\e_j Zu^\e X^\e_i Zu^\e  |Zu^\e|^{\beta-2} \eta^{\beta+2}|\nabla_\e u^\e|^2
\notag\\
& =- \beta  \sum_{i,j=1}^{2n+1}\int_{t_1}^{t_2} \int_\Om A_{i\xi_j}(x, \nabla_\e u^\e) X^\e_j Zu^\e  |Zu^\e|^{\beta-2}Zu^\e 
X^\e_i \bigg(  \eta^{\beta+2} |\nabla_\e u^\e|^2\bigg) 
\notag\\
& - \beta
\sum_{i=1}^{2n+1}\int_{t_1}^{t_2} \int_\Om
A^\e_{i, x_{2n+1}}(x, \nabla_\e u^\e) X^\e_i \bigg(|Zu^\e|^{\beta-2}Zu^\e    \eta^{\beta+2} |\nabla_\e u^\e|^2\bigg)
\notag
\\
& =-\beta (\beta+2) \sum_{i,j=1}^{2n+1}\int_{t_1}^{t_2} \int_\Om A_{i\xi_j}(x, \nabla_\e u^\e) X^\e_j Zu^\e  |Zu^\e|^{\beta-2}Zu^\e 
X^\e_i\eta \eta^{\beta+1} |\nabla_\e u^\e|^2 
\notag\\
& -2\beta \sum_{\ell,i,j=1}^{2n+1} \int_{t_1}^{t_2} \int_\Om A_{i\xi_j}(x, \nabla_\e u^\e) X^\e_j Zu^\e  |Zu^\e|^{\beta-2}Zu^\e \eta^{\beta+2} 
X^\e_\ell u^\e X^\e_i X^\e_\ell u^\e
\notag\\
& - \beta(\beta-1)
\sum_{i=1}^{2n+1}\int_{t_1}^{t_2} \int_\Om
A^\e_{i, x_{2n+1}}(x, \nabla_\e u^\e) |Zu^\e|^{\beta-2} X^\e_iZu^\e    \eta^{\beta+2} |\nabla_\e u^\e|^2
\notag\\
& - \beta(\beta+1)
\sum_{i=1}^{2n+1}\int_{t_1}^{t_2} \int_\Om
A^\e_{i, x_{2n+1}}(x, \nabla_\e u^\e) |Zu^\e|^{\beta-2} Zu^\e    \eta^{\beta+1}X^\e_i\eta |\nabla_\e u^\e|^2 
\notag\\
& - \beta
\sum_{i,\ell=1}^{2n+1}\int_{t_1}^{t_2} \int_\Om
A^\e_{i, x_{2n+1}}(x, \nabla_\e u^\e) |Zu^\e|^{\beta-2} Zu^\e    \eta^{\beta+2}X^\e_\ell u^\e X^\e_i X^\e_\ell u^\e 
\notag\\
& =I^5 + \cdots + I^9. 
\notag
\end{align}}
We observe that the ellipticity condition yields
\begin{align*}
& \beta(\beta-1)\sum_{i,j=1}^{2n+1}\int_{t_1}^{t_2} \int_\Om A_{i\xi_j}(\nabla_\e u^\e) X^\e_j Zu X^\e_i Zu^\e     |Zu^\e|^{\beta-2} \eta^{\beta+2} |\nabla_\e u^\e|^2 
\\
& \ge
(\beta+1)^2 C_{\lambda} \int_{t_1}^{t_2} \int_\Om (\ddelta+|\nabla_\e u^\e|^2)^{\frac{p-2}{2}} |\nabla_\eps Zu^\e|^2   |Zu^\e|^{\beta-2} \eta^{\beta+2} |\nabla_\e u^\e|^2.
\end{align*}
Let us now consider $I^5$:
{\allowdisplaybreaks
\begin{align*}
& I^5 = -\beta(\beta+2)\sum_{i,j=1}^{2n+1}\int_{t_1}^{t_2} \int_\Om A_{i\xi_j}(x, \nabla_\e u^\e) X^\e_j Zu^\e  |Zu^\e|^{\beta-2}Zu^\e X^\e_i\eta \eta^{\beta+1} |\nabla_\e u^\e|^2 
\\
& \leq 2(\beta+1)^2\int_{t_1}^{t_2} \int_\Om (\ddelta+|\nabla_\e u^\e|^2)^{\frac{p-1}{2}} 
|\nabla_\e Zu^\e|  |Zu^\e|^{\beta-1}|\nabla_\e\eta| \eta^{\beta+1} |\nabla_\e u^\e|^2
\\
& \leq \alpha(\beta+1)^2\int_{t_1}^{t_2} \int_\Om (\ddelta+|\nabla_\e u^\e|^2)^{\frac{p-2}{2}} |\nabla_\e Zu^\e|^2  |Zu^\e|^{\beta-2} \eta^{\beta+2} |\nabla_\e u^\e|^2 
\\
& + C_\al (\beta+1)^2\int_{t_1}^{t_2} \int_\Om (\ddelta+|\nabla_\e u^\e|^2)^{\frac{p}{2}}  |Zu^\e|^{\beta}|\nabla_\e\eta| \eta^{\beta}.
\end{align*}}
The estimate of $I^6$ is identical to that $I^1_\ell$ and we thus omit it.  Let us consider $I^7$. One has
{\allowdisplaybreaks
\begin{align*}
& I^7 \leq (\beta+1)^2
\int_{t_1}^{t_2} \int_\Om (\ddelta+|\nabla_\e u^\e|^2)^{\frac{p-1}{2}}|Zu^\e|^{\beta-2} |\nabla_\e Zu^\e|    \eta^{\beta+2} |\nabla_\e u^\e|^2
\\
& \leq \alpha(\beta+1)^2
\int_{t_1}^{t_2} \int_\Om (\ddelta+|\nabla_\e u^\e|^2)^{\frac{p-2}{2}}|Zu^\e|^{\beta-2} |\nabla_\e Zu^\e|^2    \eta^{\beta+2} |\nabla_\e u^\e|^2
\\
& + C_\al(\beta+1)^2
\int_{t_1}^{t_2} \int_\Om (\ddelta+|\nabla_\e u^\e|^2)^{\frac{p}{2}}|Zu^\e|^{\beta-2}   \eta^{\beta+2} |\nabla_\e u^\e|^2. 
\end{align*}}
Similar consideration holds for $I^8$
{\allowdisplaybreaks
\begin{align*}
& I^8 \leq (\beta+1)^2
\int_{t_1}^{t_2}  \int_\Om
(\ddelta+|\nabla_\e u^\e|^2)^{\frac{p-1}{2}} |Zu^\e|^{\beta-1}    \eta^{\beta+1} |\nabla_\e\eta| |\nabla_\e u^\e|^2
\\
& \leq C_{\Lambda}(\beta+1)^2
\int_{t_1}^{t_2}  \int_\Om
(\ddelta+|\nabla_\e u^\e|^2)^{\frac{p}{2}} |Zu^\e|^{\beta}    \eta^{\beta}|\nabla_\e\eta|^2 
\\
& + C_{\Lambda}(\beta+1)^2
\int_{t_1}^{t_2}  \int_\Om
(\ddelta+|\nabla_\e u^\e|^2)^{\frac{p+2}{2}} |Zu^\e|^{\beta-2}    \eta^{\beta+2}.
\end{align*}}
Finally, we estimate $I^9$.
{\allowdisplaybreaks
\begin{align*}
& I^9 \leq C_{\Lambda}(\beta+1)\sum_{\ell,i=1}^{2n+1}
\int_{t_1}^{t_2} \int_\Om
(\ddelta+|\nabla_\e u^\e|^2)^{\frac{p-1}{2}} |Zu^\e|^{\beta-1}    \eta^{\beta+2} |\nabla_\e u^\e| |X^\e_i X^\e_\ell u^\e| 
\\
& \leq C_{\Lambda}(\beta+1)\sum_{\ell,i=1}^{2n+1}
\int_{t_1}^{t_2} \int_\Om
(\ddelta+|\nabla_\e u^\e|^2)^{\frac{p}{2}} |Zu^\e|^{\beta-2}    \eta^{\beta+2}  |X^\e_i X^\e_\ell u^\e|^2.
\end{align*}} 
It follows that 
{\allowdisplaybreaks
\begin{align}\label{3.9}
& \sum_{k=1}^4\sum_{\ell=1}^{2n} I_\ell^k + \sum_{k=5}^9 I^k
\leq
\al\int_{t_1}^{t_2} \int_\Om \eta^{\beta+2} (\ddelta+|\nabla_\e u^\e|^2)^{\frac{p-2}{2}} |Zu^\e|^\beta \sum_{i,j=1}^{2n+1} |X^\e_iX^\e_j u^\e|^2  
\\
& + C_\al (\beta+1)^2(1+ |\nabla_\e \eta||_{L^\infty}^2)\int_{t_1}^{t_2} \int_\Om (\eta^{\beta} +  \eta^{\beta+4}) (\ddelta+|\nabla_\e u^\e|^2)^{\frac{p}{2}} |Zu^\e|^{\beta-2} \sum_{i,j=1}^{2n+1} |X^\e_iX^\e_j u^\e|^2
\notag\\
& + \frac{2\al}{(1 + ||\nabla_\e \eta||^2)(\beta + 2)}\int_{t_1}^{t_2}\int_\Omega |Zu^\e|^{\beta+2} 
\eta^{\beta+3} |\partial_t \eta| dx + \frac{\al}{(\beta + 2)^2}\int_\Omega |Zu^\e|^{\beta+2} \eta^{\beta+4}\Big|_{t=t_1}
\notag\\
& +\alpha(\beta+1)^2\int_{t_1}^{t_2} \int_\Om (\ddelta+|\nabla_\e u^\e|^2)^{\frac{p-2}{2}} |\nabla_\e Zu^\e|^2  |Zu^\e|^{\beta-2} 
(\eta^{\beta+2}+\eta^{\beta+4}) |\nabla_\e u^\e|^2
\notag\\
& + C_{\Lambda}(\beta+1)^2
\int_{t_1}^{t_2}  \int_\Om
(\ddelta+|\nabla_\e u^\e|^2)^{\frac{p+2}{2}} |Zu^\e|^{\beta-2}    \eta^{\beta+2}.
\notag
\end{align}}
Summing up equations (\ref{tXestimate}) and (\ref{tZestimate}),  we obtain
{\allowdisplaybreaks
\begin{align*}
& \sum_{i,j=1}^{2n+1}
\int_{t_1}^{t_2} \int_\Om \eta^{\beta+2}(\ddelta+|\nabla_\e u^\e|^2)^{\frac{p-2}{2}} |Zu^\e|^\beta|X^\e_iX^\e_j u^\e|^2
 + \int_\Om
( \eta^{\beta+2}|Zu^\e|^\beta |\nabla_\e u^\e|^2 )\bigg|_{t_1}^{t_2} 
\\
& + \int_{t_1}^{t_2} \int_\Om (\ddelta+|\nabla_\e u^\e|^2)^{\frac{p-2}{2}} |\nabla_\e Zu^\e|^2   |Zu^\e|^{\beta-2} \eta^{\beta+2}|\nabla_\e u^\e|^2  
\\ 
& = \int_{t_1}^{t_2} \int_{\Om} |Zu^\e|^{\beta}  |\nabla_\e u^\e|^2  \p_t ( \eta^{\beta+2}) + \sum_{k=1}^4\sum_{\ell=1}^{2n}I_\ell^k 
+\sum_{k=5}^9 I^k.
\end{align*}}
Applying \eqref{3.9}, the proof is completed.

\end{proof}

At this point we make use of the non-degeneracy condition $\delta>0$, and recalling that $Z$ is obtained as a commutator of the horizontal vector fields and that $\eta\le 1$, we estimate
\begin{equation}\label{need delta>0 }
\int_{t_1}^{t_2}\int_\Omega |Zu^\e|^{2} \eta^3  dx dt \le  C_\delta\int_{t_1}^{t_2}\int_\Om
(\ddelta+|\nabla_\e u^\e|^2)^{\frac{p-2 }2} \sum_{i,j=1}^{2n+1} |X^\e_i X_j^\e u^\e|^2 \eta^2 dx dt.
\end{equation}
Lemma \ref{lemma3.4} and \eqref{need delta>0 } yield the following

\begin{cor} \label{lemma3.4bis} Let $u^\e$ be a weak solution of \eqref{approx1} in $Q$. For any $t_2\ge t_1\ge 0$,   and all $\eta \in C^\infty_0(\Omega)$, such that $\eta \le1$, 
$ ||\partial_t\eta||\leq C ||\nabla_\e\eta||^2$. For every fixed value of $\ddelta$ 
there exists $C_\ddelta$ depending on $\delta, p,n$ and on the structure constants, such that 
\begin{align*}
& \frac{1}{2}\int_\Omega ( (\ddelta+ |\nabla_\e u^\e|^2)\eta^2)\Big|_{t_1}^{t_2} + 
\lambda\int_{t_1}^{t_2}\int_\Om
(\ddelta+|\nabla_\e u^\e|^2)^{\frac{p-2 }2} \sum_{i,j=1}^{2n+1} |X^\e_i X_j^\e u^\e|^2 \eta^2
\\
& \leq 
 C_\ddelta\int_{t_1}^{t_2}\int_\Om (\ddelta+|\nabla_\e u^\e|^2)^{p/2}\Big(\eta^2 + |\nabla_\e \eta|^2 + |\eta Z\eta|\Big).
 \end{align*}
\end{cor}

\begin{cor}\label{nablagrande} Let $u^\e$ be a  solution of \eqref{approx1} in $\Om\times (0,T)$ and 
$B_{\e}(x_0,r) \times (t_0-r^2, t_0)$ a parabolic cylinder. 
Let  $\eta\in C^{\infty}( B_{\e}(x_0,r) \times (t_0-r^2, t_0))$
be a non-negative test function $\eta\leq 1,$ which vanishes on the parabolic boundary
and such that  there exists a constant $C_{\lambda, \Lambda}>1$ for which
$||\partial_t \eta||_{L^\infty}\leq C_{\lambda, \Lambda} (1 + ||\nabla_\e\eta||^2_{L^\infty}).$
 Set $t_1 = t_0-r^2 $.  
There exists a constant $C_{\ddelta, \lambda, \Lambda}$, also depending on $\ddelta$, such that 
for all $\beta\ge 2$  one has
{\allowdisplaybreaks
\begin{align}\label{corlongwaydown}
& \int_{t_0-r^2}^{t_0} \int_\Om \eta^{\beta+2}(\ddelta+|\nabla_\e u^\e|^2)^{\frac{p-2}{2}} |Zu^\e|^\beta \sum_{i,j=1}^{2n+1} |X^\e_iX^\e_j u^\e|^2 + \max_{t \in (t_0-r^2, t_0]} \int_\Om
 \eta^{\beta+2}|Zu^\e|^\beta |\nabla_\e u^\e|^2 
\\ 
& +  (\beta+ 1)^2\int_{t_0-r^2}^{t_0} \int_\Om (\ddelta+|\nabla_\e u^\e|^2)^{\frac{p-2}{2}} |\nabla_\e Zu^\e|^2   |Zu^\e|^{\beta-2} \eta^{\beta+2} 
|\nabla_\e u^\e|^2
\notag\\
& \leq C_{\lambda, \Lambda}  (\beta+1)^2(1+ |\nabla_\e \eta||_{L^\infty}^2)\int_{t_0-r^2}^{t_0} \int_\Om \eta^{\beta}  (\ddelta+|\nabla_\e u^\e|^2)^{\frac{p}{2}} |Zu^\e|^{\beta-2} \sum_{i,j=1}^{2n+1} |X^\e_iX^\e_j u^\e|^2
\notag\\
& + C_{\lambda, \Lambda} (\beta+1)^2
\int_{t_0-r^2}^{t_0}  \int_\Om
(\ddelta+|\nabla_\e u^\e|^2)^{\frac{p+2}{2}} |Zu^\e|^{\beta-2}    \eta^{\beta+2}.
\notag
\end{align}}
\end{cor}

\begin{proof} The statement follows at once by standard parabolic pde arguments, after choosing $\alpha$ appropriately small in  \eqref{longwaydown} and applying  \eqref{need delta>0 }, once one notes that $|Zu^\e|\le  \sum_{i,j=1}^{2n+1} |X^\e_i X^\e_j u^\e|$.

\end{proof}

\begin{cor}\label{CCGMc2.7} 
In the hypotheses of the previous corollary we have 
\begin{align*}
& \sum_{i,j=1}^{2n+1}\int_{t_0}^{t_0-r^2} \int_\Om \eta^{\beta+2}(\ddelta+|\nabla_\e u^\e|^2)^{\frac{p-2}{2}} |Zu^\e|^\beta|X^\e_iX^\e_j u^\e|^2 + \max_{t \in (t_0-r^2, t_0]} \int_\Om
 \eta^{\beta+2}|Zu^\e|^\beta |\nabla_\e u^\e|^2
\\  
& + (\beta+1)^2 \int_{t_0}^{t_0-r^2} \int_\Om (\ddelta+|\nabla_\e u^\e|^2)^{\frac{p-2}{2}} |\nabla_\e u^\e|^2   |Zu^\e|^{\beta-2} \eta^{\beta+2} (X^\e_lu^\e)^2 
 \\
 & \le C^{\beta/2}(\beta+1)^\beta ( ||\nabla_\e \eta ||^2_{L^\infty} +1))^{\beta/2} \sum_{i,j=1}^{2n+1}\int_{t_0}^{t_0-r^2} \int_{\Om}
 \eta^\beta \Big(
\ddelta + |\nabla_\e u^\e|^2 \Big)^{\frac{p-2+\beta}{2}}
|X^\e_i X^\e_j u^\e|^2,
\end{align*}
where $c = c(n, p,L) > 0$.
\end{cor}

\begin{proof} In order to handle the 
first term in the right-hand side of the sought for conclusion, it suffices to observe that 
\begin{align*}
& C(\beta+1)^2 ( ||\nabla_\e \eta ||^2_{L^\infty} +1)
 \eta^\beta \Big(
\ddelta + |\nabla_\e u^\e|^2 \Big)^{p/2}
|Zu|^{\beta -2}\sum_{i,j=1}^{2n+1}|X^\e_i X^\e_j u^\e|^2 =
\\
& =
 \eta^{\beta-2} \Big(
\ddelta + |\nabla_\e u|^2 \Big)^{(p-2)(\beta-2)/2 \beta}
|Zu|^{\beta -2}(|\sum_{i,j=1}^{2n+1}X^\e_i X^\e_j u^\e|^2 )^{(\beta -2)/\beta} 
\\
& + \eta^{2} \Big( \ddelta + |\nabla_\e u^\e|^2 \Big)^{(p+\beta -2)/\beta}
(\sum_{i,j=1}^{2n+1}|X^\e_i X^\e_j u^\e|^2 )^{2/\beta}
C(\beta+1)^2 ( ||\nabla_\e \eta ||^2_{L^\infty} +1).
\end{align*}
The conclusion then follows from  H\"older's inequality.
We can handle the second term in the same way 
{\allowdisplaybreaks
\begin{align*}
& C(\beta+1)^2 ( ||\nabla_\e \eta ||^2_{L^\infty} )
 \eta^\beta \Big(
\ddelta + |\nabla_\e u^\e|^2 \Big)^{(p+2)/2}
|Zu|^{\beta -4}\sum_{i,j=1}^{2n+1}|X^\e_i X^\e_j u^\e|^2
\\
& =
 \eta^{\beta-2} \Big(
\ddelta + |\nabla_\e u|^2 \Big)^{\frac{(p-2)(\beta-4)}{2 \beta}}
|Zu|^{\beta -4}(|\sum_{i,j=1}^{2n+1}X^\e_i X^\e_j u^\e|^2 )^{(\beta - 4)/\beta} 
\\
& \times \eta^{2} \Big( \ddelta + |\nabla_\e u^\e|^2 \Big)^{2(p+\beta -2)/\beta}
(\sum_{i,j=1}^{2n+1}|X^\e_i X^\e_j u^\e|^2 )^{4/\beta}
C(\beta+1)^2 ( ||\nabla_\e \eta ||^2_{L^\infty} +1).
\end{align*}}

\end{proof}

The key step in the proof of the Lipschitz regularity of solutions is the following Caccioppoli type inequality which is a parabolic analogue of \cite[Theorem 3.1]{Zhong}.

\begin{thrm}\label{CCGMt2.8} 
Let $u^\e$ be a  solution of \eqref{approx1} in $\Om\times (0,T)$ and 
$B_{\e}(x_0,r) \times (t_0-r^2, t_0)$ a parabolic cylinder. 
Let  $\eta\in C^{\infty}( B_{\e}(x_0,r) \times (t_0-r^2, t_0])$
be a non-negative test function $\eta\leq 1,$ which vanishes on the parabolic boundary
such that  there exists a constant $C_{\lambda, \Lambda}>1$ for which
$||\partial_t \eta||_{L^\infty}\leq C_{\lambda, \Lambda} (1 + ||\nabla_\e\eta||^2_{L^\infty}).$
 Set $t_1 = t_0-r^2, t_2=t_0$. There exist constants $C,K>0$ depending on $\ddelta$ such that 
for all $\beta\ge 2$  one has
{\allowdisplaybreaks
\begin{align*}
& \int_{t_1} ^{t_2} \int_\Omega
\eta^2(\ddelta + |\nabla_\e u^\e|^2)^{(p-2+\beta)/2} \sum_{i,j=1}^{2n+1} |X^\e_i X^\e_j u^\e|^2 dx dt
 + \frac{1}{\beta+2} \max_{t \in (t_0-r^2, t_0]}\int_\Omega  (\ddelta+ |\nabla_\e u^\e|^2)^{\frac{\beta}{2}+1}\eta^2
\\
& \leq C (\beta + 1)^5 (||\nabla_\e\eta||^2_{
L^{\infty}} + ||\eta Z\eta||_{L^{\infty}} + 1)
\int_{t_1} ^{t_2} \int_{spt(\eta)}
(\ddelta + |\nabla_\e u^\e|^2)^{(p+\beta)/2}.
\end{align*}}
Here, $C$ depends only on $ p,$ and $\Lambda$.
\end{thrm}

\begin{proof}  In view of Lemma \ref{lemma3.4}, the conclusion will follow once we provide an appropriate estimate of the
term
{\allowdisplaybreaks
\begin{align*}
\int_{t_1} ^{t_2} \int_\Omega
\eta^2(
\ddelta + |\nabla_\e u^\e|^2)^{(p-2+\beta)/
2} |Zu^\e|^2.
\end{align*}}
The first step is to apply H\"older's inequality to obtain
{\allowdisplaybreaks
\begin{align*}
& \int_{t_1} ^{t_2}\int_\Omega
\eta^2(
\ddelta + |\nabla_\e u^\e|^2)^{(p-2+\beta)/
2} |Zu^\e|^2 
\\
& \leq\Big(
\int_{t_1}^{t_2}\int
\eta^{\beta +2}(
\ddelta + |\nabla_\e u^\e|^2)^{\frac{p-2}{2}}
 |Zu^\e|^{\beta +2} dx dt\Big)^{\frac{2}{\beta +2}}\Big(\int_{t_1}^{t_2}
\int_{spt(\eta)}
(
\ddelta + |\nabla_\e u^\e|^2)^{\frac{p+\beta }{2}}\Big)^{
\frac{\beta }{\beta +2}}
\\ 
& (\text{since}\ |Zu^\e| \leq \sum_{i,j=1}^n | X_i^\e X_j^\e u^\e|)
\\
& \leq\Big(\int_{t_1}^{t_2}
\int
\eta^{\beta +2}(
\ddelta + |\nabla_\e u^\e|^2)^{\frac{p-2}{2}}
 |Zu^\e|^{\beta}\sum_{i,j=1}^n | X_i^\e X_j^\e u^\e|^2\Big)^{\frac{2}{
\beta +2}}\Big(\int_{t_1}^{t_2}
\int_{spt(\eta)}
(
\ddelta + |\nabla_\e u^\e|^2)^{\frac{p+\beta }{2}}\Big)^{
\frac{\beta }{\beta +2}}
\\
& (\text{the first integral in the right-hand side can be bounded by applying Corollary \ref{CCGMc2.7}, resulting in the estimate})
\\
& \le C^{\frac{\beta}{\beta+2}}(\beta+1)^{\frac{2\beta }{\beta +2}} ( ||\nabla_\e \eta ||^2_{L^\infty} +  1)^{\frac{\beta }{\beta +2}} \Big( \int_{t_1}^{t_2} \int_{\Om}
 \eta^\beta \Big(
\ddelta + |\nabla_\e u^\e|^2 \Big)^{\frac{p-2+\beta}{2}}
\sum_{i,j=1}^{2n+1}|X^\e_i X^\e_j u|^2\Big)^{
\frac{2}{
\beta +2}}
\\
& \times \Big(\int_{t_1}^{t_2}
\int_{spt(\eta)}
(
\ddelta + |\nabla_\e u^\e|^2)^{\frac{p+\beta}{2}}\Big)^{
\frac{\beta }{\beta +2}}
\\
& (\text{by Young' s inequality, recalling}\ C_0\ \text{from the statement of Lemma \ref{lemma3.4}}) 
\\
& \le C \frac{\beta}{\beta+2} \Big(\frac{4 C_0(\beta + 1)^4}{(\beta +2)}\Big)^{ \frac{2}{\beta}} (\beta+1)^{2} (||\nabla_\e \eta ||^2_{L^\infty} +  1)
\int_{t_1}^{t_2}
\int_{spt(\eta)}
(
\ddelta + |\nabla_\e u^\e|^2)^{\frac{p+\beta}{2}}
\\
& + \frac{1}{2 C_0(\beta+1)^4}\int_{t_1}^{t_2} \int_{\Om}
 \eta^\beta \Big(
\ddelta + |\nabla_\e u^\e|^2 \Big)^{\frac{p-2+\beta}{2}}
\sum_{i,j=1}^{2n+1}|X^\e_i X^\e_j u^\e|^2.
\end{align*}}
Now we note that 
$$\frac{\beta}{\beta+2} \Big(\frac{4 C_0(\beta + 1)^4}{(\beta +2)}\Big)^{ \frac{2}{\beta}} (\beta+1)^{2} \leq C_{\lambda, \Lambda} (\beta +1)^5. $$
Substituting the previous estimate in Lemma \ref{lemma3.4}, we conclude
{\allowdisplaybreaks
\begin{align*}
& \int_{t_1} ^{t_2} \int_\Omega
\eta^2(\ddelta + |\nabla_\e u^\e|^2)^{\frac{p-2+\beta}{2}}|X^\e_i X^\e_j u^\e|^2 
\\
& \leq C_{\lambda, \Lambda}(\beta+1)^{5} (||\nabla_\e \eta ||^2_{L^\infty} + ||\eta Z \eta ||_{L^\infty} + 1)
\int_{t_1} ^{t_2} \int_{spt(\eta)}
(\ddelta + |\nabla_\e u^\e|^2)^{\frac{p+\beta}{2}}.
\end{align*}}
This completes the proof of the theorem.
 
\end{proof}

In the next result, from Lemma \ref{SobXU} and Theorem \ref{CCGMt2.8} we will establish local integrability of $\nabla_\e u^\e$ in  $L^q$ for every $q\ge p$.

\begin{lemma}\label{GainDu}
Let $u^\e$ be a solution of \eqref{approx1} in $Q$. 
For any open ball $B\subset \subset \Om$ and $T>t_2\ge t_1\ge 0$, consider a test  function 
$\eta\in C^{^\infty}([0,T] \times B)$, vanishing on the 
parabolic boundary,  such that $\eta \le1$, 
$ ||\partial_t\eta||\leq C ||\nabla_\e\eta||^2$.
For every $\beta\ge 0 $, there exists a constant 
$C=C(n, p, \lambda, \Lambda, d(B, \p\Om), T-t_2, \ddelta)>0$, 
such that 
\begin{align*}
& \int_{t_1} ^{t_2}\int_\Omega
(\ddelta + |\nabla_\e u^\e|^2)^{(\beta+p+2)/2} |\eta|^{\beta+2}
\leq C^\beta (\beta+1)^\beta
\int_{t_1} ^{t_2}\int_B
(\ddelta + |\nabla_\e u^\e|^2)^{p/2}.
\end{align*}
\end{lemma}

\begin{proof}
We begin by examining the case $\beta=0$. Applying Lemma \ref{SobXU} and Corollary \ref{lemma3.4bis} one can find positive constants $C_1,C_2, C_3$, depending on $n, p, \lambda, \Lambda, d(B, \p\Om), T-t_2, \ddelta$, such that
\begin{align*}
& \int_{t_1} ^{t_2}\int_\Omega
(\ddelta + |\nabla_\e u^\e|^2)^{(p+2)/2} |\eta|^{2} \leq 
C _1(p +1)^2 \int_{t_1} ^{t_2}\int_\Omega
(\ddelta+  |\nabla_\e u^\e|^2)^{\frac{p-2}{2}} \sum_{i,j}| X_j^\e X_i^\e u^\e|^2 |\eta|^{ 2}
\\
& + C_2\beta^2 \int_{t_1} ^{t_2}\int_\Omega
(\ddelta + |\nabla_\e u^\e|^2)^{p/2} ( |\eta|^{2}
+ |\nabla_\e \eta |^2) \le
C_3\int_{t_1}^{t_2}\int_\Om (\ddelta+|\nabla_\e u^\e|^2)^{p/2}\Big(\eta^2 + |\nabla_\e \eta|^2 + |\eta Z\eta|\Big),
\end{align*}
concluding the proof in the case $\beta=0$.
%%%%%%%%%%%
Next, we consider the range $\beta\ge 2$. The interpolation inequality Lemma \ref{SobXU} and   Theorem \ref{CCGMt2.8} imply the existence of positive constants $C_4,..., C_7$, depending on $n, p, \lambda, \Lambda, d(B, \p\Om), T-t_2,$ and $ \ddelta$, 
such that 
\begin{align}\label{caccio3-bis}
& \int_{t_1} ^{t_2}\int_\Omega
(\ddelta + |\nabla_\e u^\e|^2)^{(\beta+p+2)/2} |\eta|^{\beta+2} 
\\
& \leq 
C_4(\beta +p +1)^2 \int_{t_1} ^{t_2}\int_\Omega
(\ddelta+  |\nabla_\e u^\e|^2)^{\frac{p+\beta-2}{2}} \sum_{i,j}| X_j^\e X_i^\e u^\e|^2 |\eta|^{\beta + 2}
\notag\\
& + C_5\beta^2 \int_{t_1} ^{t_2}\int_\Omega
(\ddelta + |\nabla_\e u^\e|^2)^{(\beta+p)/2} |\eta|^{\beta} ( |\eta|^{2}
+ |\nabla_\e \eta |^2)
\notag\\
& \leq C_6 (\beta +p+ 1)^7 
\int_{t_1} ^{t_2} \int_{\Om}
(\ddelta + |\nabla_\e u^\e|^2)^{(p+\beta)/2} \Big(\eta^2 + |\nabla_\e \eta|^2 + |\eta Z\eta|\Big) 
\notag\\
& \le
C_7 (\beta+1)^7 (||\nabla_\e \eta ||^2_{L^\infty} + ||\eta Z \eta ||_{L^\infty} + 1)
\int_{B}
(\ddelta + |\nabla_\e u^\e|^2)^{(p+\beta)/2}.
\notag
\end{align}
Iterating the latter $[\beta]/2$ times, the conclusion follows.

\end{proof}

In the next result we establish Lipschitz bounds that are uniform in $\e$. The argument consists in implementing Moser iterations, and rests on the observation  that  the quantity  $\ddelta + |\nabla_\e u^\e|^2$ is bounded from below by $\delta>0$, and that for every $\beta\ge 0$  it
is bounded in $L^{p+\beta}$ in a parabolic cylinder, uniformly in $\e$.

In the iteration itself, we will consider metric balls $B_\e$ defined through the Carnot-Caratheodory metric associated to the Riemannian structure $g_\e$ defined by the orthonormal frame $X_1^\e,...,X_{2n+1}^\e$. We recall here that $g_\e$ converges to the sub-Riemannian structure of the Heisenberg group in the Gromov-Hausdorff sense \cite{Koranyi}, and in particular $B_\e\to B_0$ in terms of Hausdorff distance. These considerations should make it clear that the estimates in the following theorem are stable as $\e\to 0$.

\begin{thrm}\label{moser}
Let $u^\e$ be a  solution of \eqref{approx1} in $\Om\times (0,T)$ and 
$Q_0^\e= B_{\e}(x_0,r) \times (t_0-r^2, t_0)$ a parabolic cylinder contained in $\Om\times (0,T)$. For given $\sigma \in (0,1)$, there exists a constant $C=C(p,\sigma,\beta_0,\lambda,\Lambda, \ddelta)>0$ such that 
\begin{equation}\label{moser-conclusion}
\sup_{B(x_0,\sigma r)\times (t_0-(\sigma r)^2,t_0)}  (\ddelta+|\nabla_\e u^\e|^2)^{\frac{p}{2}}
\le C 
\fint_{t_0- r^2}^{t_0} \fint_{B(x_0, r)} (\ddelta+|\nabla_\e u^\e|^2)^{\frac{p}{2}}.
\end{equation}
\end{thrm}

\begin{proof}
We recall the main steps. 
Let  us consider a family of cylinders 
$Q^\e_i= B_{\e}(x_0, r_i) \times (t_0-r_i^2, t_0)\subset\subset Q_0^\e$ and with $r_i<r_{i-1}$. 
Applying (ii) in Lemma \ref{lemma-sobolev} to the function 
$w_\beta= (\ddelta+ |\nabla_\e u^\e|^2)^{\frac{\beta+2}{4}}$, 
one obtains 
\begin{align*}
& \Bigg(\int_{t_0-r_i^2} ^{t_0} \int_{B_\e(x_0, r_i)}
(\ddelta +|\nabla_\e u^\e|^2)^{ \frac{ (\beta+2) N_1}{2(N_1-2)} } \Bigg)^\frac{N_1 - 2}{N_1}=
||w_\beta||^2_{\frac{2N_1}{N_1-2}, \frac{2N_1}{N_1-2}, Q^\e_i}
\\
& \le  ||w_\beta||^2_{2,\infty, Q^\e_i }+||\nabla_\e w_\beta||^2_{2,2,Q^\e_{i}}  
\\
& \le \int_{t_0-r_i^2} ^{t_0} \int_{B_\e(x_0, r_i)}
\eta^2(\ddelta + |\nabla_\e u^\e|^2)^{\beta/2} \sum_{i,j=1}^{2n+1} |X^\e_i X^\e_j u^\e|^2 
 + \frac{1}{\beta+2} \max_{t \in (t_0-r^2, t_0]}\int_{B_\e(x_0, r_i)} (\ddelta+ |\nabla_\e u^\e|^2)^{\frac{\beta+2}{2}}\eta^2. 
\end{align*}
Next, we set $g=(\ddelta + |\nabla_\e u^\e|^2)^{(p-2)/2}$. 
Using Theorem \ref{CCGMt2.8}, along with the fact that $(\ddelta+|\nabla_\e u^\e|)\ge \ddelta>0$, we obtain
$$\Bigg(\int_{t_0-r_i^2} ^{t_0} \int_{B_\e(x_0, r_i)}
(\ddelta+|\nabla_\e u^\e|^2)^{ \frac{ (\beta+2) N_1}{2(N_1-2)} } \Bigg)^\frac{N_1 - 2}{N_1}\le 
\frac{C(\beta+p)^6 }{(r_i-r_{i-1})^2} 
\int_{t_0-r_i^2} ^{t_0} \int_{B_\e(x_0, r_i)} g 
(\ddelta + |\nabla_\e u^\e|^2)^{ (\beta+2)/2}.
$$
Setting $q =  \frac{(\beta + 2) N_1}{N_1-1}$ and $k= \frac{N_1-1}{N_1-2} $ in the latter inequality, we deduce
\begin{align*}
& \Bigg(  \fint_{t_0-r_i^2}^{t_0}   \fint_{B(x_0,r_i)}
(\sqrt{\ddelta +|\nabla_\e u^\e|^2})^{q k}\Bigg)^\frac{1}{q\, k}
\\
& \le 
C^{\frac{1}{\beta+2}}(\beta+p)^{\frac{6}{\beta+2}} \Bigg(  \frac{r_{i-1}^{2+N}}{r_i^{N}(r_i-r_{i-1})^2}
\Bigg)^{\frac{1}{2+\beta}}
  \Bigg(  \fint_{t_0-r_{i-1}^2}^{t_0}  \fint_{B(x_0,r_{i-1})}
g (\sqrt{\ddelta +|\nabla_\e u^\e|^2})^{\beta +2 }\Bigg)^{\frac{1}{\beta + 2}}
\\
& \le 
C^{\frac{1}{\beta+2}}(\beta+p)^{\frac{6}{\beta+2}} \Bigg(  \frac{r_{i-1}^{2+N}}{r_i^{N}(r_i-r_{i-1})^2}
\Bigg)^{\frac{1}{2+\beta}}
  \Bigg(\fint_{t_0-r_{i-1}^2}^{t_0}  \fint_{B(x_0,r_{i-1})}
(\sqrt{\ddelta +|\nabla_\e u^\e|^2})^{q}\Bigg)^\frac{1}{q}.
\end{align*}
The classical Moser iteration
scheme in see \cite{moser} now applies, leading to the sought for conclusion.

\end{proof}

%%%%%%%%%%%%%%%%%%%%%%%%%%%%%%%%%%%%%%%%%%%%%%%%%%%%%%%%%%%%%%%

\section{H\"older regularity of derivatives of $u^\e$.}\label{S:hold}

This section focuses on  the proof of the second part of Theorem \ref{main epsilon}. I.e., we want to show that  for each $\ddelta, \e>0$ a weak solution $$u^\e\in L^p ((0,T), W^{1,p,\e}(\Om))\cap C^2(Q)$$  of the approximating PDE \eqref{approx1} in $Q=\Om\times (0,T)$ satisfies the H\"older estimates 
\[
||\nabla_\e u^\e ||_{C^{\alpha} (B\times (t_1,t_2) )} + ||Zu^\e||_{C^{\alpha} (B\times (t_1,t_2) )}\le C
\bigg(\int_0^T \int_\Om (\ddelta+|\nabla_\e u^\e|^2)^{\frac{p}{2}} dx dt \bigg)^{\frac{1}{p}},
\]
for any open ball $B\subset \subset \Om$ and $T>t_2\ge t_1\ge 0$,
and for some 
constants $C=C(n, p, \lambda, \Lambda, d(B, \p\Om), T-t_2, \ddelta)>0$  and 
$\al=\al (n, p, \lambda, \Lambda, d(B, \p\Om), T-t_2, \ddelta) \in (0,1)$ independent of $\e$. It is clear that the above estimate represents the $\e$-version of \eqref{c1alpha}.

We begin by studying  the regularity of the derivatives of $u^\e$ along the second layer of the Lie algebra of $\Hn$. 
First of all, we observe that  since $\ddelta>0$ is fixed,  Lemma \ref{GainDu}   and Theorem \ref{CCGMt2.8} imply that  for all $i,j=1,...,2n$, one has $|X_iX_j u^\e|$ is bounded in $L^2$ uniformly in $\e>0$.  It follows that  $Zu^\e\in L^2_{loc}(Q)$ uniformily in $\e>0$. 
In view of Lemma \ref{eqZ} we can actually obtain more.

\begin{prop}\label{Zu-reg} 
Let $u^\e$ be a  solution of \eqref{approx1} in $\Om\times (0,T)$ and 
$Q= B(x_0,r) \times (t_0-r^2, t_0)$ a parabolic cylinder contained in $\Om\times (0,T)$. There exists  constants $C=C(p,\sigma,\beta_0,\lambda,\Lambda, \ddelta)>0$ and $\al=\al(p,\sigma,\beta_0,\lambda,\Lambda, \ddelta)\in (0,1)$ such that 
 $$||Zu^\e||_{C^{\al}(Q)} + ||XZu^\e||_{L^2(Q)} \le C \bigg(||u^\e||_{L^p(2Q)}+ ||\nabla_\e u^\e||_{L^p(2Q)}\bigg).$$
 \end{prop}
\begin{proof}
In view of Theorem \ref{moser}, we observe that  $|\nabla_\e u^\e|$ is bounded 
and recalling \ref{eqZ}, one deduces  that for each $\e>0$ the smooth function $w^\e=Zu^\e$ satisfies the PDE
\begin{equation}\label{Z-eq}
\partial_t w^\e=
\sum_{i=1}^{2n+1}
X^\e_i\bigg( \sum_{j=1}^{2n+1}
a_{ij}^\e(x,t) X^\e_j w^\e +f^\e_i(x,t)\bigg),
\end{equation}
where
$$a_{ij}^\e(x,t) = A^\e_{i, \xi_j}(x, \nabla_\e u^\e)\ \ \ \text{and}\ \ \  f^\e_i(x,t)=A^\e_{i, x_{2n+1}}(x, \nabla_\e u^\e),
$$
are locally bounded in $Q$,  unformly in $\e>0$, and for every $\eta\in \R^{2n+1}$ and for a.e. $x\in \Omega$ satisfy
$$\lambda |\eta|^2 \le \sum_{i,j=1}^{2n+1} a^\e_{ij}(x,t) \eta_i \eta_j \le \Lambda |\eta|^2.$$
Thanks to Theorems \ref{CCGMt2.8} and \ref{moser},  and observing that $\delta>0$, one immediately infers that 
 $w^0=Zu$ is locally in $L^2$. But then, the Caccioppoli inequality implies  that $\nabla_\e Zu^\e$ is in $L^2_{loc}(Q)$, uniformly in $\e>0$.  The Harnack inequality established in \cite{CCR} and \cite{ACCN} yields interior H\"older estimates for $w^\e$ in $Q$, which are stable as  $\e\to 0$.
 \end{proof}
 
 \begin{rmrk} Actually, a stronger result holds. Let $\alpha\in (0,1)$ denote the H\"older exponent of $Zu^\e$ (which is uniform in $\e>0$). By observing that $w^\e-w^\e(x_0,t_0)$ is also a solution of \eqref{Z-eq}, then a standard Caccioppoli type argument yields
 \begin{equation}\label{Zu in Morrey}\int_{t_0-r^2}^{t_0}\int_{B} |\nabla_\e Zu |^2 dx dt \le C \frac{1}{r^{2}}\int_{t_0-(2r)^2}^{t_0}\int_{2B} |w^\e-w^\e(x_0,t_0)|^2 dx dt 
 \le C r^{\alpha-2} r^{2n+2+2}.\end{equation}
This shows, in particular, that $|\nabla_\e Zu|$ belongs to the parabolic Morrey class $M^{2,\alpha/2}(Q)$, where for $\lambda\in (0,1)$ and $q\ge 1$ we have indicated with $M^{q,\lambda}(Q)$ the space of all functions $f\in L^q(Q)$ such that for all $B\subset \Omega$, and $0<t_0<T$, one has
\begin{equation}\label{morrey-d}
\sup_{r>0} r^{-(2n+4)}\int_{ \min(t_0-r^2,0)}^{t_0}
\int_{B\cap \Omega}  |f|^q dx dt \le Cr^{q(\lambda-1)}. 
\end{equation}
\end{rmrk}
We also recall that the parabolic Campanato spaces $\mathscr{L}^{q,\lambda}(Q)$ is the collection of all $f\in L^q(Q)$ such that for all $B=B(x_0,r) \subset \Omega$, and $0<t_0<T$, one has
\begin{equation}\label{campanato-d}
\sup_{r>0} r^{-(2n+4)}\int_{ \min(t_0-r^2,0)}^{t_0}
\int_{B\cap \Omega}  |f-f_{(x_0,t_0),r}|^q dx dt \le Cr^{q(\lambda-1)}. 
\end{equation}
Here, we have set $$f_{(x_0,t_0)} = r^{-(2n+4)}\int_{ \min(t_0-r^2,0)}^{t_0}
\int_{B\cap \Omega} f(x,t) dx dt.$$
A standard argument, see for instance \cite{daprato}, shows that the Campanato space is isomorphic to the space of H\"older continuous functions. In particular, we rely on the following instance of this general result.

\begin{lemma}\label{campanato-holder}
Let $K\subset \subset Q$. There exists $M,r_0>0$ such that for any $(x_0,t_0)\in K$ and $0<r<r_0$,  if $f \in \mathscr{L}^{q,\lambda}(B(x_0,r)\times (t_0-r^2,t_0))$ then $f\in C_\e^\lambda (B(x_0,r/M)\times (t_0-r^2/M^2,t_0))$.
\end{lemma}

Next, we return to the study of the regularity of horizontal derivatives of solutions.  By virtue of  Lemma \ref{diff-PDE} we recall that if  $u^\e$ is a solution of \eqref{approx1} in $Q$, if for a fixed $\ell=1,...,2n$ we set $v^\e = X^\e_\ell u^\e$, and $s_\ell = (-1)^{[\ell/n]}$, then the function $v^\e$ is  a  solution in $Q$ of
\begin{equation}\label{eq-X}
\p_t v^\e = \sum_{i=1}^{2n+1} X^\e_i\Big(  \sum_{j=1}^{2n+1}a^\e_{ij}(x,t) X^\e_j v^\e + a^\e_i(x,t)\Big) + a^\e(x,t), 
\end{equation}
where
$$a_{ij}^\e(x,t)=  A_{i, \xi_j}^\e (x, \nabla_\e u^\e)\in L^\infty_{loc}(Q), $$
$$a_i^\e(x,t)=  A_{i, x_\ell}^\e (x, \nabla_\e u^\e) - \frac{s_\ell x_{\ell+s_{\ell}n}}{2} A_{i, x_{2n+1}}^\e (x, \nabla_\e u^\e) \in L^\infty_{loc}(Q) ,$$
and
$$
a^\e(x,t)= s_\ell Z (A_{\ell+s_{\ell}n}^\e (x, \nabla_\e u^\e)),$$
with $$ |a^\e(x,t)| \le C |\nabla_\e Zu^\e|\in L^2_{loc}(Q)\cap M^{2,\alpha}(Q)$$
for some constant $C>0$ depending only on the structure constants and on $||u||_{L^{p,p}(Q)}.$
We need to invoke a standard result from the theory of Morrey-Campanato which adapts immediately to the Heisenberg group setting,  see \cite{morrey}, \cite{cap99}. In the statement of the next lemma we assume that $Q=\Omega\times (0,T)$ is a given cylinder, and that $[b_{ij}]_{i,j = 1}^{2n+1}$ is a uniformly elliptic matrix-valued function in $Q$, with coefficients in $L^\infty(Q)$. We also suppose that for some $\lambda \in (0,1)$ we are given functions $b_i\in M^{2,\lambda}(Q)$, and a function $b$ such that for each $2B\subset \Omega$ and $r>0$ sufficiently small, \begin{equation}\label{morrey-hyp}   r^{-N_1}\int_{ \min(t_0-r^2,0)}^{t_0}
\int_{B\cap \Omega}  |b| dx dt \le Cr^{\lambda-2} \end{equation} with $b\in L^{2N_1/(N_1+2)}_{loc}(Q)$, where we recall that $N_1=N+2=2n+4$ is the parabolic dimension with respect to the dilations $(x,t)\to (\delta_\lambda x,\lambda^2 t)$.

\begin{lemma}\label{morreycampanato} For each $\e\ge 0$, let $w\in  L^p((0,T), W_0^{1,p}(\Om))$ be a weak solution in $Q$  to the equation 
$$\p_t w = \sum_{i=1}^{2n+1} X_i^\e \bigg( \sum_{j=1}^{2n+1} b_{ij} (x,t) X_j^\e w + b_i(x,t)\bigg) + b(x,t).$$
Then, $|\nabla_\e w|\in M^{2,\lambda}\bigg(\frac{1}{2}B\times (t_0-(\frac{r}{2})^2, t_0) \bigg).$
\end{lemma}

We can now conclude the proof of the second part of Theorem \ref{main epsilon}. To begin, as we need to  apply  Lemma \ref{morreycampanato} to the linear equation  \eqref{eq-X},   we observe that  \eqref{Zu in Morrey} and H\"older inequality yield the needed hypothesis \eqref{morrey-hyp}.  At this point one can invoke Lemma \ref{morreycampanato} to conclude that for every $\ell=1,...,2n$ the function $\nabla_\e X_l^\e u^\e$ belongs locally to $M^{2,\lambda}$.  In view of the Poincar\'e inequality,  one then has that  $\nabla_\e u^\e$ belongs to the Campanato spaces $\mathscr{L}^{2,\lambda}$ and hence  by virtue of Lemma \ref{campanato-holder} it is H\"older continuous, concluding the proof.

\bigskip

\section{Proof of Theorem \ref{main1}}\label{S:main}

We will   need a simple form of the comparison principle, see \cite{Bieske1} and \cite{Bieske2}.
\begin{lemma}\label{comparison}
Let $u,w$ be weak solutions of \eqref{maineq} in a cylinder $B\times (t_1,t_2)$. If on the parabolic boundary $ B\times\{t_1\} \cup \p B \times (t_1,t_2)$ one has that $u\ge w$, then $u\ge w$ in $B\times (t_1,t_2)$.
\end{lemma}
We now show how Theorem \ref{main1} follows from the comparison principle and from Theorem \ref{main epsilon}.

\begin{proof}[of Theorem \ref{main1}]
Recall from Lemma \ref{alpha} that  $u$ is H\"older continuous in any compact subdomain of $Q$, in particular in the closure of $B\times (t_1,t_2)$. For each $\e>0$ consider $u^\e$,  the  unique smooth solution of the quasilinear parabolic problem
\begin{equation}\label{ivp}
\Bigg\{ \begin{array}{ll}  \p_t u^\e  = \sum_{i=1}^{2n+1} X_i^\e A_i^\e (x,\nabla_\e u^\e), & \text{ in } B\times(t_1,t_2) \\
u^\e =u &\text{ in }  B\times\{t_1\} \cup \p B \times (t_1,t_2), \end{array}
\end{equation}
where $A_i^\e(x,\xi)$ satisfies the structure conditions \eqref{structure-epsilon}. By virtue of Theorem \ref{main-epsilon-lebesgue} and of the H\"older regularity from Theorem \ref{main epsilon}, one has that  for every $K\subset \subset Q$, and $q\ge 1$, there exist $M=M(p, q,\lambda,\Lambda,n,\delta)>0$  and $\al=\al(p, q,\lambda,\Lambda,n,\delta)\in (0,1)$, such that for every $\e>0$, $(x_0,t_0)\in K$ and $B(x_0,r)\times (t_0-r^2,t_0) \subset Q$, 
\begin{align*}
& ||\nabla_\e |\nabla_e u^\e|^q ||_{L^2(B(x_0,r)\times (t_0-r^2,t_0))}  \le M,
\\
& ||Z| \nabla_\e  u^\e|^q ||_{L^2(B(x_0,r)\times (t_0-r^2,t_0))}  \le M 
\\
& ||\nabla_\e u^\e||_{C^\al_\e (B(x_0,r)\times (t_0-r^2,t_0))} +
||Z  u^\e||_{C^\al_\e (B(x_0,r)\times (t_0-r^2,t_0))}   \le M.
\end{align*}
By the theorem of Ascoli-Arzel\`a, one can find  $u_0\in C^{1,\al}_{loc}(Q)$ and a sequence $\e_k\to 0$ such that 
$$u^{\e_k}\to u_0  \text{ and } \nabla_{\e_k} u^{\e_k} \to \nabla_0 u_0\ \text{ uniformly on compact subsets of Q.}$$

The latter implies that $u_0$ is a weak solution of \eqref{maineq}, in $B(x_0,r)\times (t_0-r^2,t_0)$, which agrees with the function $u$ on the parabolic boundary of $B(x_0,r)\times (t_0-r^2,t_0)$. By the comparison principle, the solution to this boundary values problem is unique, and hence we conclude that $u\in C^{1,\al}_{loc}(B(x_0,r)\times (t_0-r^2,t_0))$.

\end{proof}

\end{document}